\documentclass[reqno,twoside]{amsart}
\usepackage[top=4cm, bottom=4cm, left=3cm, right=3cm]{geometry}
\usepackage{graphics}
 \usepackage{graphicx}
 \usepackage{epsfig}
 \usepackage{amsmath}
\usepackage{amsfonts}
\usepackage{amssymb}
\usepackage{xcolor}
\usepackage{enumerate}
\usepackage{commath,bm}

\usepackage{subcaption}
\usepackage[outercaption]{sidecap}  
\usepackage{float}
\usepackage{wrapfig}
\usepackage{eucal}
\usepackage{mathrsfs,mathtools,epic,commath,bm}

\newtheorem{thm}{Theorem}[section]
\newtheorem{corollary}[thm]{Corollary}
\newtheorem{proposition}[thm]{Proposition}
\newtheorem{theorem}[thm]{Theorem}
\newtheorem{lemma}[thm]{Lemma}

\newtheorem{definition}[thm]{Definition}
\newtheorem{remark}[thm]{Remark}

\numberwithin{equation}{section}

\newcommand{\riemannDerivative}{D}
\newcommand{\genericGroup}{G}
\newcommand{\auxfunct}{{\widetilde{g}}}
\newcommand{\auxangle}{\gamma}
\newcommand{\finalfunct}{g}
\newcommand{\finalangle}{\beta}
\newcommand{\RR}{\mathbb{{R}}}

\newcommand{\NN}{\mathbb{{N}}}

\newcommand{\CC}{\mathbb{{C}}}

\newcommand{\BSect}{\text{\normalfont{BSect}}}

\newcommand{\EOVarphi}{\mathcal{E}_0 (BS_{\varphi,a})}
\newcommand{\EOmega}{\mathcal{E} [BS_{\omega,a}]}
\newcommand{\MOmega}{\mathcal{M} [BS_{\omega,a}]}

\title[Generalized Black-Scholes PDEs]{Introducing and solving generalized Black-Scholes PDEs through the use of functional calculus}

\author{Jes\'us Oliva-Maza}
\address{J. Oliba-Maza: Departamento de Matematicas, Instituto Universitario de Matematicas y Aplicaciones, Universidad de Zaragoza, Zaragoza 50009, Spain}
\email{joliva@unizar.es}

\author{Mahamadi Warma}
\address{M. Warma: Department of Mathematical Sciences and the Center for Mathematics and Artificial Intelligence (CMAI),
	George Mason University, Fairfax, VA 22030, USA}
\email{mwarma@gmu.edu}

\thanks{This work has been done when the first author was visiting the Department of Mathematical Sciences at George Mason University in Fairfax, Virginia. He is grateful for the financial aid of Warma's AFOSR research grant. The work of J. Oliva-Maza  is partially supported by Project PID2019-105979GB-I00 of the MICINN, and by BES-2017-081552, MINECO, Spain and FSE. The work of M. Warma is partially supported by AFOSR under Award NO:  FA9550-18-1-0242 and US Army Research Office (ARO) under Award NO: W911NF-20-1-0115.}

\subjclass[2010]{26A33, 47A60, 47B12,  47D06}

\keywords{Bisectorial operators, sectorial operators,  functional calculus, holomorphic semigroups, Riemann-Liouville space-fractional derivative,  Weyl space-fractional derivative, generalized  Black-Scholes equations}

\begin{document}

\begin{abstract}
We introduce some families of generalized Black--Scholes equations which involve the Riemann-Liouville and Weyl space-fractional derivatives. We prove that these generalized Black--Scholes equations are well-posed in $(L^1-L^\infty)$-interpolation spaces. More precisely, we show that the elliptic type operators involved in these equations generate holomorphic semigroups. Then, we give explicit integral expressions for the associated solutions.  In the way to obtain well-posedness,  we prove a new connection between bisectorial operators and sectorial operators in an abstract setting. Such a connection extends some known results in the topic to a wider family of both operators and the functions involved.
\end{abstract}

\maketitle

\section{Introduction}

Bisectorial operators play a central role in the theory of abstract inhomogeneous first order differential equations on the whole real line, like
\begin{equation}\label{rl}
	u'(t) = Au(t) + f(t),\qquad t\in\RR,
\end{equation}
where $A$ is a bisectorial operator on a Banach space $X$. 
The theory has been an active topic of research during the past years. We refer to \cite{arendt2010decomposing, schweiker2000asymptotics} and their references for a complete overview, the importance, and applications of such operators. 

It is well-known that sectorial operators are related to the theory of abstract first order homogeneous differential equations on the positive real axis with initial conditions. Namely, equations of the form
\begin{equation}\label{prl}
	v'(t) = Av(t),\quad t>0,\qquad  v(0)=g.
\end{equation}
We refer to the monographs  \cite{ arendt2011vector, engel2000one} and their references for a precise description of this fact. 

Both families of operators are closely related.  Indeed,  an operator $A$ is bisectorial if and only if both $aI+A$ and $aI-A$ are sectorial for some real number $a\geq 0$, where $I$ denotes the identity operator. As one may expect,  both families share multiple properties and, in particular, one can define similar functional calculus for the two families of operators.
Indeed, in this work we make use of the meromorphic functional calculus from sectorial operators completely developed in the excellent monograph \cite{haase2006functional} and we adapt it carefully to the theory of  bisectorial operators, which extends the functional calculi considered in \cite{morris2010local} and the references therein.  


Another connection between these two families lies in the fact that $A^2$ is sectorial whenever $A$ is a bisectorial operator, see e.g.  \cite[Proposition 5.1]{arendt2010decomposing}. In the particular case that $A$ generates an exponentially bounded group, then $A^2$ generates a holomorphic semigroup (see for instance \cite[Theorem 1.15]{arendt1986one} or \cite[Corollary 4.9]{engel2000one}). An application of this result is to study differential equations on the positive real line like Equation \eqref{prl} in terms of a possibly simpler equation on the real line Equation  \eqref{rl}. One may find a concrete example of this fact in \cite{arendt2002spectrum}, where the authors obtained properties of the classical Black--Scholes equation
\begin{equation}\label{BSEquationIntro} \tag{BS}
	u_t = x^2 u_{xx} + x u_x, \qquad t,x >0,
\end{equation}
through the simpler and elegant partial differential equation 
$$u_t = -xu_x,\qquad t,x > 0.$$
Since the seminal work \cite{black1973pricing}, the Black--Scholes equation \eqref{BSEquationIntro} has been an active topic of research in mathematical finance due to its importance in the modeling of pricing options contracts, see for instance \cite{gozzi1997generation} and the references therein. One of our purposes in the present paper is to study fractional differential equations which extend the classical equation \eqref{BSEquationIntro}, using the theory of bisectorial operators established here. At this point, we wish to observe that part of our contribution is to show how the relations between bisectorial operators and sectorial operators that we develop here can be used successfully to solve several of those equations extending \eqref{BSEquationIntro}.  We are not dealing with mathematical finance in this paper, but on the other hand we are confident that our new generalized Black--Scholes partial differential equations could be used to understand the disturbing and anomalous behavior of the financial market.


Let us explain the method we follow to extend \eqref{BSEquationIntro} to a fractional differential equation. As stated before, the classical Black--Scholes equation is studied by means of the following degenerate differential operator:
\begin{equation}\label{opA}
	(Jf)(x) := -xf'(x),\qquad x>0,
\end{equation}
on $(L^1-L^\infty)$ interpolation spaces that we shall explain in Section \ref{BlackScholesSection}.  In \cite{arendt2002spectrum}, the authors used the connection between the operator $J$ and the classical Cesàro operator $\mathcal{C}$ given by
\begin{align*}
	(\mathcal{C} f)(x) &= \frac{1}{x} \int_0^x f(y) \, dy, \quad x>0
\end{align*}
and its adjoint Cesàro operator $\mathcal{C}^\ast$. This connection had been first pointed out in \cite{cowen1984subnormality} to study the Cesàro operator $\mathcal{C}$ on the half-plane. In addition, the differential operator $J$ was also related in \cite{lizama2014boundedness} to the generalized fractional version of the Ces\`aro operator $\mathcal{C}_\alpha$ on $L^p$-spaces, for real numbers $\alpha > 0$,  given by
\begin{equation}\label{al-cesao}
	(\mathcal{C}_\alpha f)(x) = \frac{\alpha}{x^\alpha} \int_0^x (x-y)^{\alpha - 1}f(y) \, dy = \frac{\Gamma(\alpha+1)}{x^\alpha} (\riemannDerivative^{-\alpha}f) (x) , \quad x>0
\end{equation}
and the associated adjoint Ces\`aro operator $\mathcal{C}_\alpha^\ast$ given by
\begin{equation}\label{ad-cesaro}
	(\mathcal{C}_\alpha^\ast f)(x) = \alpha \int_x^\infty \frac{(y-x)^{\alpha - 1}}{y^\alpha}f(y) \, dy =  \Gamma(\alpha+1) (W^{-\alpha}(y^{-\alpha}f)) (x), \quad x>0,
\end{equation}
where $\riemannDerivative^{-\alpha}$ and $W^{-\alpha}$ denote the Riemann-Liouville and Weyl fractional integrals of order $\alpha$,  respectively, see Section \ref{notations} for more details. We notice that $\riemannDerivative^{-\alpha}$ is usually denoted by $I^{-\alpha}$ in most of the literature on fractional calculus. The connection between $J$ and $\mathcal{C}_\alpha$, $\mathcal{C}_\alpha^\ast$ is given in terms of the exponentially bounded group $(G(s))_{s\in \RR}$ generated by $J$, with expression $(G(s)f)(x) = f(e^{-s}x)$. More precisely, we have that 
\begin{equation}
	\begin{aligned}\label{FC1}
		(\mathcal{C}_\alpha f)(x)  &= \alpha \int_0^\infty e^{-s} (1-e^{-s})^{\alpha-1} (G(s)f)(x)\, ds, \quad x >0,
		\\ (\mathcal{C}_\alpha^\ast f)(x) & = \alpha \int_{-\infty}^0 (1-e^{s})^{\alpha-1} (G(s)f)(x)\, ds, \quad x > 0.
	\end{aligned}
\end{equation}

Furthermore, \eqref{FC1} yields a representation in functional calculus of the type $\mathcal{C}_\alpha = f_\alpha(J)$, $\mathcal{C}_\alpha^\ast = f_\alpha^\ast(J)$ for some suitable holomorphic functions $f_\alpha,\, f_\alpha^\ast$ which have some fractional-powers like behavior that we shall specify later in the paper. Therefore,  as in the same way one can write \eqref{BSEquationIntro} in terms of the operators $\mathcal{C}, \,\mathcal{C}^\ast$, it seems natural to construct some families of generalized Black--Scholes equations which can be written in terms of $\mathcal{C}_\alpha,\, \mathcal{C}_\alpha^\ast$, operators which,  respectively involve fractional Riemann-Liouville derivatives of order $\alpha$, $\riemannDerivative^\alpha$, and fractional Weyl derivatives of order $\alpha$, $W^\alpha$. This method will give rise to generalized fractional Black--Scholes equations of the following three forms for $x$, $t>0$:
\begin{equation}\label{GBSE}
	\begin{aligned}
		u_t & = \frac{1}{\Gamma(\alpha+1)^2} \riemannDerivative^\alpha (x^\alpha \riemannDerivative^\alpha (x^\alpha u)) - \frac{2}{\Gamma(\alpha+1)} \riemannDerivative^\alpha (x^\alpha u) + u, \\
		u_t &= \frac{1}{\Gamma(\alpha+1)^2} x^\alpha W^\alpha (x^\alpha W^\alpha u),
		\\ u_t & = - \frac{1}{\Gamma(\alpha+1)^2} \riemannDerivative^\alpha (x^{2\alpha} W^\alpha u)+\frac{1}{\Gamma(\alpha+1)}x^\alpha W^\alpha u ,
	\end{aligned}
\end{equation}
for suitable values of $\alpha$, see Section \ref{BlackScholesSection} for more details. The theory of bisectorial operators developed here will allow us to obtain the well-posedness and an explicit integral expressions of solutions of these fractional versions of the Black--Scholes equation, of course by incorporating initial and boundary conditions.  Also, at the limiting case $\alpha=1$, we recover all the classical known results.

	In order to get these results, we first give a significant extension of the above mentioned results regarding a bisectorial operator $A$ and its square $A^2$.  Indeed, we are able to prove that if $A$ is a bisectorial operator and $g$ is a suitable function whose range is contained in a sector, and which has a fractional-power like behavior at the singularity points of the spectra of $A$ and $g(A)$, then $g(A)$ is a sectorial operator (see Theorem \ref{sectorialTheorem}). 
	As one may expect, the setting for sectorial operators serves again as an inspiration for this result, more precisely the scaling property, i.e., for a sectorial operator $A$, $A^\alpha$ is sectorial for small enough $\alpha>0$, see \cite[Theorem 2]{kato1960note}. We are also able to give supplementary results of special importance when $g(A)$ generates a semigroup, such as the characterization of the closure of its domain $\overline{\mathcal{D}(g(A))}$ or an integral expression for the semigroup it generates in terms of the functional calculus of $A$.
	
	Interestingly, as a consequence of our results one obtains that, if $A$  is bisectorial of angle $\frac{\pi}{2}$ and half-width $a\geq 0$ (in particular, if it generates an exponentially bounded group), then either $(A+a)^\alpha$ or $-(A+a)^\alpha$ generate a holomorphic semigroup for all $\alpha > 0$ with $\alpha \neq 1,3,5,\ldots$ This is a remarkable discovery that generalizes the already known results in this directions for the case $\alpha = 2$, see for example \cite[Theorem 1.15]{arendt1986one} and \cite[Proposition 5.1]{arendt2010decomposing}, or in the case where $A$ generates a bounded group, see \cite[Theorem 4.6]{baeumer2009unbounded}.

	As a final remark, we mention that fractional versions of the Black--Scholes equation have been proposed and analyzed in a number of papers, see for instance \cite{fall2019black, kumar2012analytical, song2013solution, zhang2016numerical}. In most of these references,  authors are only able to deal with time-fractional derivatives, that is, replacing the time derivative $u_t$ in \eqref{BSEquationIntro} with a Riemann-Liouville or a Caputo type time fractional derivative.  Of course in this setting the authors cannot get some generation of semigroups results due to the limitation of time-fractional derivatives. Spatial-fractional derivatives are indeed more difficult to deal with,  since spatial terms are more complex than the time ones in the equation \eqref{BSEquationIntro}. A spatial fractional Black--Scholes model can be found in \cite{chen2014analytically} without given further details. The  fractional Black--Scholes equation we propose here also contains fractional powers acting as multiplication, yielding to equations which are definitely difficult to solve by more classical methods such as the Laplace transform or the Fourier transform. 
	Therefore, we observe that the fractional versions of \eqref{BSEquationIntro} proposed in the present paper,  which are solved through the theory we developed here, seem to be notably difficult to be solved with the classical methods.
	
	In short, the contributions of the present paper can be regarded as centered around two facts:
	\begin{enumerate}
		\item The introduction of new (generalized) fractional Black--Scholes equations arising from fractional Cesàro operators in a natural manner. As we have already observed, such equations are difficult -maybe not possible- to solve by classical methods.
		\item In order to overcome the quoted failure of usual methods, we establish a new connection, in an abstract setting, between bisectorial operators and sectorial operators. Such a connection extends notably previous results in the field. Actually our approach is based on the proof of the scaling property given in \cite[Proposition 5.2]{auscher1997holomorphic}, but it requires quite more general functions to operate in functional calculi defined on the basis of more intricate integration paths, as well as nontrivial, more involved, approximation tools.  This requires more work to be proven than the latter one. This is because we consider more general functions that require working with approximation tools, and moreover our family of functional calculus presents more intricate integration paths.
	\end{enumerate}
	
	The  rest of the paper is organized as follows.  In Section \ref{sect21} we fix some notations and introduce some definitions.  In Section \ref{Sec2} we develop the functional calculus of bisectorial operators and its extensions. Section \ref{sectorialSection} is devoted to the passage from bisectorial to sectorial operators, where we also give a note regarding holomorphic functional calculus for sectorial operators. In Section \ref{Sec4} we give some generation of bounded holomorphic semigroups results by the  general sectional operators we have constructed in Section \ref{sectorialSection}. The complete theory and applications of the generalized Black--Scholes equation  are contained in Section \ref{BlackScholesSection}. We conclude the paper with \ref{Appendix} where some useful results have been proven.

	\section{Notations and preliminary results} \label{notations}
	
	In this section, we fix some notations,  give some definitions and introduce some well-known results as they are used throughout the remainder of the paper.
	We start with some notations and definitions.
	
	\subsection{Notations and definitions}\label{sect21}
	Given any $\varphi\in (0,\pi)$, we denote the sector
	\begin{align}\label{sector}
		S_{\varphi}:=\Big\{z\in\CC: \left|\mbox{arg}(z)\right|<\varphi\Big\},
	\end{align}
	and set $S_0 := (0,\infty)$.  If $\lambda\in\CC$, we shall let $\Re \lambda$ denote the real part of $\lambda$.
	
	
	Let $A$ with domain $\mathcal{D}(A)$ be a closed linear operator in a Banach space $X$. We shall denote by $\sigma(A)$ the spectrum of $A$, $\sigma_p(A)$ the point spectrum of $A$,  and by $\rho(A):=\CC\setminus\sigma(A)$, the resolvent set of $A$.  We also let $\widetilde\sigma(A):=\sigma(A)\cup\{\infty\}$.
	For $\lambda\in\rho(A)$ we shall let $R(\lambda,A):=(\lambda I-A)^{-1}$ be the resolvent operator of $A$. Then,  $R(\lambda,A)$ is a bounded linear operator from $X$ into $X$. By $\mathcal R(A)$, we shall mean the range space of $A$ and $\mathcal N(A)$ shall denote the null set (or the kernel) of the operator $A$.  
	
	We shall denote by ${\mathcal L(X)}$ the space of all bounded linear operators from a Banach space $X$ into $X$.  If $B\in\mathcal L(X)$, we shall let $\|B\|_{\mathcal L(X)}$ be the operator norm of $B$. 
	
	Finally, we use throughout the notation $h\lesssim g$ to denote $h\le M g$, for some constant $M$,  when the dependence of the constant $M$ on some physical parameters is not relevant, and so it is suppressed.
	

	\begin{definition}
		A closed linear operator $A$ with domain $\mathcal D(A)$ in a Banach space $X$ is said to be a {\bf sectorial operator} if there exists $\varphi \in [0,\pi)$ such that $\sigma(A)\subset \overline{S_{\varphi}}$ and, for every $\varphi' \in (\varphi, \pi)$, there exists a constant $M_{\varphi'}>0$ such that
		\begin{align*}
			\|\lambda R(\lambda,A)\|_{\mathcal L(X)}\le M_{\varphi'}\;\mbox{ for all }\lambda\in \CC\setminus S_{\varphi'}.
		\end{align*}
	\end{definition}

	Next,  for any $\omega \in (0,\pi/2]$ and $a\geq 0$ a real number,  we set 
	\begin{equation*}
		BS_{\omega,a} :=
		\begin{cases}
			(-a + S_{\pi-\omega}) \cap (a - S_{\pi-\omega}) \;\;&\mbox{ if } \omega\in (0,\pi/2)  \mbox{ and }  a\ge 0,\\
			\\
			i\RR &\mbox{ if } \omega =\pi/2 \mbox{ and } a=0.
		\end{cases}
	\end{equation*}
	
	
	\begin{figure}[h]
		\centering
		\includegraphics[width=1\textwidth]{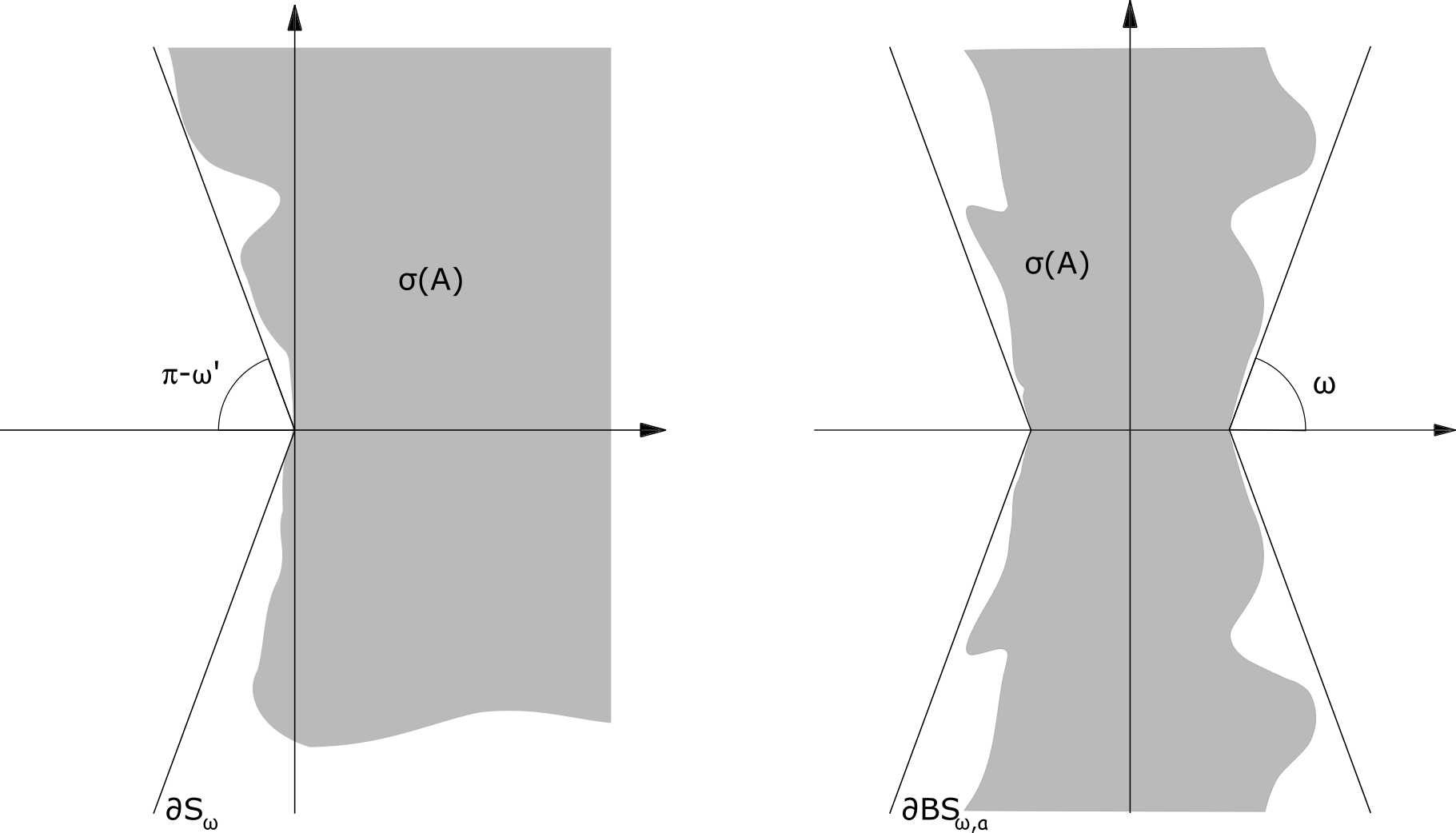}
		\caption{Illustration of the spectrum of a sectorial operator (left) and bisectorial operator (right).}
		\label{fig1}
	\end{figure}
	
	\begin{definition}
		Let $(\omega,a) \in (0,\pi/2] \times [0,\infty)$ and let $A$ be a closed linear operator in a Banach space $X$.  We will say that $A$ is a {\bf bisectorial operator} of angle $\omega$ and half-width $a$ if the following conditions hold:
		\begin{itemize}
			\item $\sigma (A) \subset \overline{BS_{\omega,a}}$.
			\item For all $\omega' \in (0, \omega)$, $A$ satisfies the resolvent bound
			\begin{align*}
				\sup \Big\{ \min\{|\lambda-a|, |\lambda+a|\} \| R(\lambda,A)\|_{\mathcal L(X)} \, : \, \lambda \notin \overline{BS_{\omega',a}} \Big\}< \infty .
			\end{align*}
		\end{itemize}
	\end{definition}

	Given a Banach space $X$, we will denote the set of all bisectorial operators of angle $\omega$ and half-width $a$ in $X$ by $\BSect(\omega,a)$. We omit an explicit mention to $X$ for the sake of simplicity in our notations.  Notice that a closed operator $A \in \BSect(\omega,a)$ if and only if both $aI+A$ and $aI-A$ are sectorial of angle $\pi-\omega$. 

	Next, we introduce the notion of bounded holomorphic semigroups in the sense of  \cite[Definition 3.7.3]{arendt2011vector}.
	
	
	\begin{definition}\label{semigroupDef}
		Let $\delta \in (0,\pi/2]$.  A mapping $T : S_\delta \to \mathcal{L}(X)$  is called a {\bf bounded holomorphic semigroup} (of angle $\delta$) if it has the following properties:
		\begin{enumerate}[(a)]
			\item The semigroup property:  $T(z)T(z') =T (z + z')$ for all $z,  z' \in S_\delta$.
			\item The mapping $T : S_\delta \to \mathcal{L}(X)$ is holomorphic.
			\item The mapping $T$ satisfies: $\sup_{w\in S_{\delta'}} \|T(w)\|_{\mathcal L(X)}< \infty$ for every $\delta' \in (0, \delta)$.
		\end{enumerate}
	\end{definition}
	
	It is known that every bounded holomorphic semigroup $T$ on a Banach space $X$  is determined by its generator $A$, which is a closed linear operator given by
	\begin{equation}\label{genInf}
		\begin{cases}
			\displaystyle \mathcal D(A)=\left\{x\in\ X: \lim_{t \downarrow 0} \frac{T(t)x - x}{t} \;\mbox{exists}\right\}\\
			\displaystyle  A x = \lim_{t \downarrow 0} \frac{T(t)x - x}{t}, \quad x\in \mathcal D(A).
		\end{cases}
	\end{equation} 
	
	Another way to define the infinitesimal generator is through its resolvent $R(\lambda, A)$ by using the Laplace transform. Namely,
	\begin{align}\label{resolventOfaSemigroup}
		R(\lambda, A) = \int_0^\infty e^{-\lambda t} T(t) \, dt, \quad \Re \lambda >0.
	\end{align}
	
	For any $\delta \in (0,\pi/2]$, it is well known that a closed linear operator $A$ is the generator of a bounded holomorphic semigroup of angle $\delta$ if and only if $-A$ is a sectorial operator of angle $\pi/2-\delta$. Even more, one has that
	\begin{align*}
		T (w) = \exp_{-w} (A), \quad w \in S_\delta,
	\end{align*}
	where we make use of the primary functional calculus of sectorial operators and where $\exp_{-w}(z):= \exp (-wz)$. 

	
	\begin{definition}
For $\delta \in (0,\pi/2]$, we say that a mapping $T : S_\delta \to \mathcal{L}(X)$ is an  {\bf exponentially bounded holomorphic semigroup} if it satisfies the bounded semigroup property, is holomorphic (namely (a) and (b) in Definition \ref{semigroupDef} are satisfied) and it holds that the set $\{\|T(w)\|_{\mathcal L(X)} :w \in S_{\delta'}, \, |w| \leq 1 \}$ is bounded for every $\delta' \in (0,\delta)$. 
		
		In particular, for each $\delta' \in (0,\delta)$ one can find $M_{\delta'}, \rho_{\delta'}\geq 0$ such that $\|T(w)\|_{\mathcal L(X)} \leq M_{\delta'} e^{\rho_{\delta'} \Re w}$ for all $w \in S_{\delta'}$. As a consequence,  one can define the generator $A$ of an exponentially bounded holomorphic semigroup $T$ in the same way as in \eqref{genInf} or \eqref{resolventOfaSemigroup} for all $\lambda$ with $\Re \lambda$ large enough. 
	\end{definition}
	
	It is readily seen that $A$ generates an exponentially bounded holomorphic semigroup if and only if there is some $\rho \in \RR$ such that $A-\rho$ generates a bounded holomorphic semigroup (of possibly strictly smaller angle). 
	\begin{definition}
		The space $\mathbb D_T$ of strong continuity of a (holomorphic) semigroup $T$ of angle $\delta$ is defined by
		\begin{align*}
			\mathbb D_T:= \left\{x\in X:\;	\lim_{S_{\delta'} \ni w \to 0} T(w)x = x \quad \text{for all } \delta' \in (0,\delta)\right\}.
		\end{align*}
	\end{definition}
	
	It is well known that the space  $\mathbb D_T$ is precisely the closure of the domain of its generator, that is,  $\mathbb D_T=\overline{\mathcal{D}(A)}$. We refer the reader to \cite[Section 3.4]{haase2006functional} and \cite[Section II.4.a]{engel2000one} for more details about holomorphic semigroups. 
	
	
	
	Next, we introduce the notion of $(L^1-L^\infty)$-interpolation spaces. Let $E$ be a Banach space consisting of functions with domain $(0,\infty)$, for which the inclusions $(L^1(0,\infty) \cap L^\infty(0,\infty)) \subset E \subset (L^1(0,\infty) + L^\infty(0,\infty))$ hold and are continuous. We will say that $E$ is a {\bf$(L^1-L^\infty)$-interpolation space} if, for every linear operator $S: (L^1(0,\infty) + L^\infty(0,\infty)) \to (L^1(0,\infty) + L^\infty(0,\infty))$ that restricts to bounded operators $S|_{L^1(0,\infty)}:L^1(0,\infty) \to L^1(0,\infty), \, S|_{L^\infty(0,\infty)}: L^\infty(0,\infty) \to L^\infty(0,\infty)$, then it holds that the restriction to $E$, $S|_E: E\to E$, is well defined and bounded. 
	This class includes many of the classical function spaces (e.g.  $L^p$-spaces,  Orlicz spaces, Lorenz spaces, Marcinkiewiecz spaces). Also, $E$ is said to have an {\bf order continuous norm} if $\|f_n\|_E \to 0$ for every sequence of functions $f_n \in E$ converging to $0$ almost everywhere and for which $|f_n|$ is non-increasing. For more details about $(L^1-L^\infty)$-interpolation spaces, we refer to the monograph \cite{bennett1988interpolation}.

	We conclude this section by recalling the definition of the Riemann-Liouville and Weyl fractional integrals and derivatives. 
	Let $\alpha > 0$ be a real number  and $f$ a suitable function defined on $(0,\infty)$. The Riemann-Liouville fractional integral of order $\alpha$ of $f$,  denoted $\riemannDerivative^{-\alpha}f$, and the Weyl fractional integral of order $\alpha$ of $f$,  denoted $W^{-\alpha}f$, are respectively given by
	\begin{align}\label{RLFI}
		\riemannDerivative^{-\alpha}f (x) := \frac{1}{\Gamma(\alpha)} \int_0^x (x-y)^{\alpha-1} f(y) \, dy, \quad x > 0,
	\end{align}
	and
	\begin{align}\label{WFI}
		W^{-\alpha}f (x) := \frac{1}{\Gamma(\alpha)} \int_x^\infty (y-x)^{\alpha-1} f(y) \, dy, \quad x > 0,
	\end{align}
	where $\Gamma$ denotes the usual Euler-Gamma function.
	
	The Riemann-Liouville fractional derivative of order $\alpha$ of $f$, denoted $D^{\alpha}f$, and the Weyl fractional derivative of order $\alpha$ of $f$, denoted $W^\alpha f$, are respectively given by
	\begin{align}\label{RLFD}
		\riemannDerivative^\alpha f(x) := \frac{d^n}{dx^n} D^{-(n-\alpha)}f(x), \quad x > 0,
	\end{align}
	and
	\begin{align}\label{WFD}
		W^\alpha f(x) := (-1)^n\frac{d^n}{dx^n} W^{-(n-\alpha)}f(x), \quad x > 0,
	\end{align}
	where $n$ is the smallest integer greater than or equal to $\alpha$.  We refer to the monograph  \cite{samko1993fractional} and the references therein for the class of functions for which the above expressions  \eqref{RLFI}-\eqref{WFD} exist.
	

	

	\subsection{The natural functional calculus (NFC) for bisectorial operators}\label{Sec2}

	Next, we proceed to develop a functional calculus for bisectorial operators which is analogous to the functional calculus for sectorial operators presented in \cite{haase2006functional}. This functional calculus for bisectorial operators will extend the ones considered in \cite{morris2010local, schweiker2000asymptotics} and the references therein. 
	
	For any $(\varphi,a) \in (0,\pi/2] \times [0,\infty)$, we have the algebra of holomorphic functions
	\begin{align*}
		\EOVarphi:= &\Bigg\{f \in H^\infty (BS_{\varphi,a}) \, : \, \int_{\partial (BS_{\omega', a})} \frac{|f(z)|}{\min\{|z-a|,|z+a|\}} |dz| < \infty, 
		\\ & \hfill\qquad\qquad\qquad\qquad\qquad \text{for all } \varphi< \omega' \leq \pi/2 \,\text{ and } \lim_{z \to  -a, a, \infty} f(z) = 0 \Bigg\},
	\end{align*}
	where $H^\infty (\Omega)$ denotes the algebra of bounded holomorphic functions in an open subset $\Omega \subset \CC$. It is readily seen that $\EOVarphi$ is an ideal of $H^\infty (BS_{\varphi,a})$. 
	
	Given a bisectorial operator $A \in \BSect (\omega,a)$ for some $\omega > \varphi$, we define the algebraic homomorphism $\Phi: \EOVarphi \to \mathcal{L}(X)$ given by
	\begin{align*}
		\Phi(f) := f(A) := \frac{1}{2\pi i} \int_\Gamma f(z) R(z,A)\, dz, 
	\end{align*}
	where $\Gamma$ is the positively oriented boundary of the  bisector $BS_{\omega',a}$ for any $\omega' \in (\varphi, \omega)$ (see Figure 1). It is readily seen that $f(A)$ is well defined and does not depend on the choice of $\omega'$. Moreover, one obtains the following result.
	
	\begin{lemma}
		The mapping $\Phi: \EOVarphi \to \mathcal{L}(X)$ satisfies the following properties:
		\begin{enumerate}[(a)]
			\item $\Phi$ is a homomorphism of algebras.
			\item If $T\in \mathcal{L}(X)$ commutes with the resolvent operator $R(\lambda,A)$ of $A$, then it also commutes with each operator $\Phi(f)$, where $f \in \EOVarphi$.
			\item For $\lambda \notin \overline{BS_{\varphi,a}}$ and $f \in \EOVarphi$, one has
			\begin{align*}
				R(\lambda, A) \Phi(f) = \Phi(f) R(\lambda, A) = \left( \frac{f(z)}{\lambda -z}\right)(A).
			\end{align*}
			\item  If
			$$f(z)=\frac{z}{(\lambda -z)(\mu - z)}\;\mbox{ for }  \lambda, \mu \notin \overline{BS_{\varphi,a}},$$
			then $f \in \EOVarphi$ and
			\begin{align*}
				f(A) = A R(\lambda, A) R(\mu, A) \quad \text{for } \lambda, \mu \notin \overline{BS_{\varphi,a}}.
			\end{align*}
		\end{enumerate}
	\end{lemma}
	
	\begin{proof}
		The lemma is an immediate consequence of straightforward applications of  Cauchy's theorem and the resolvent identity. See \cite[Proposition 2.2]{haase2005general} for an analogous result about sectorial operators.
	\end{proof}
	
	Next,  we add some functions in order to make our functional calculus to contain certain class of resolvents. Set
	\begin{align*}
		\mathcal{E} (BS_{\varphi,a}) := \EOVarphi \oplus \CC \frac{1}{b+z} \oplus \CC \frac{1}{b-z}\;\mbox{ for any } b \in \CC \backslash \overline{BS_{\varphi,a}}. 
	\end{align*}
	It is easy to see that the definition of $\mathcal{E} (BS_{\varphi,a})$ is independent on the particular choice of $b$.
	
	Then, one can extend $\Phi$ from $\mathcal{E}_0(BS_{\varphi,a})$ to $\mathcal{E} (BS_{\varphi,a})$ by defining 
	$$\Phi\left(\frac{1}{b+z}\right) = -R(-b,A)\mbox{ and } \Phi\left(\frac{1}{b-z}\right) = R(b,A),$$
	so that $\Phi : \mathcal{E}(BS_{\varphi,a}) \to \mathcal{L}(X)$ is an algebraic homomorphism. Moreover, $(\mathcal{E} (BS_{\varphi,a}), \mathcal{M}(BS_{\varphi,a}), \Phi)$ is an abstract functional calculus (see e.g.  \cite{haase2005general}), where $\mathcal{M}(\Omega)$ denoted the algebra of meromorphic functions with domain an open set $\Omega \subset \CC$. 
	
	\begin{definition}
		Let $(\varphi,a) \in (0,\pi/2] \times [0,\infty)$. We introduce the following notions.
		\begin{enumerate}[(a)]
			\item A function $f \in \mathcal{M}(BS_{\varphi,a})$ is said to be regularizable by $\mathcal{E} (BS_{\varphi,a})$ if there exists $e \in \mathcal{E}(BS_{\varphi,a})$ such that $e(A)$ is injective and $ef \in \mathcal{E}(BS_{\varphi,a})$. 
			
			\item For any regularizable $f \in \mathcal{M}(BS_{\varphi,a})$ we set
			\begin{align*}
				f(A):= e(A)^{-1} (ef)(A).
			\end{align*}
		\end{enumerate}
	\end{definition}
	
	By \cite[Lemma 3.2]{haase2005general}, one has that this definition is independent of the regularizer $e$, and that $f(A)$ is a well defined closed operator. 
	
	Finally,  we set 
	\begin{align*}
		\mathcal{M}(BS_{\varphi,a})_A &:= \{ f \in \mathcal{M}(BS_{\varphi,a}) \, : \, f \text{ is regularizable by } \mathcal{E}(BS_{\varphi,a})\},
	\end{align*}
	and
	\begin{align*}
		H(A) &:= \{f \in \mathcal{M}(BS_{\varphi,a})_A \, : \, f(A) \in \mathcal{L}(X)\}.
	\end{align*}
	
	This meromorphic functional calculus satisfies the following properties.
	\begin{lemma}\label{functionalCalculusProperties}
		Let $(\varphi,a) \in (0,\pi/2] \times [0,\infty)$ and let also $f \in \mathcal{M}(BS_{\varphi,a})_A$ with $A \in \BSect(\omega,a)$. Then the following assertions hold.
		\begin{enumerate}[(a)]
			\item If $S\in \mathcal{L}(X)$ commutes with $A$, that is, $SA \subset AS$, then $S$ also commutes with $f(A)$, i.e. $S f(A) \subset f(A)S$.
			\item $\mathbf{1}(A) = I$ and $z(A) = A$.
			\item Let $g \in \mathcal{M}(BS_{\varphi,a})_A$. Then
			\begin{align}\label{fgA}
				f(A) + g(A) \subset (f+g)(A),\;\mbox{ and }\; f(A) g(A) \subset (f g)(A).
			\end{align}
			Furthermore, $\mathcal{D}(f(A)g(A)) = \mathcal{D}((fg)(A)) \cap \mathcal{D}(g(A))$, and one has equality in \eqref{fgA} if $g(A) \in \mathcal{L}(X)$.
			\item $\Phi : H(A) \to \mathcal{L}(X)$ is a homomorphism of unital algebras.
			\item Let $g \in H(A)$ such that $g(A)$ is bounded and injective. Then $f(A) = g(A)^{-1} f(A) g(A)$.
			\item Let $\lambda \in \CC$. Then 
			\begin{align*}
				\frac{1}{\lambda - f(z)} \in \mathcal{M}(BS_{\varphi,a})_A \iff \lambda - f(A) \text{ is injective}.
			\end{align*}
			If this is the case, $(\lambda-f(z))^{-1}(A) = (\lambda -f(A))^{-1}$. In particular, $\lambda \in \rho(A)$ if and only if $(\lambda - f(z))^{-1} \in H(A)$.
		\end{enumerate}
	\end{lemma}
	\begin{proof}
		The lemma is a direct consequence of the results contained in \cite[ Section 3 ]{haase2005general}.
	\end{proof}
	
	Since the inclusions $\mathcal{E}(BS_{\varphi,a}) \subset \mathcal{E}(BS_{\varphi',a}), \, \mathcal{M}(BS_{\varphi,a}) \subset \mathcal{M}(BS_{\varphi',a})$ hold for any $\varphi < \varphi' < \omega$, we can form the inductive limits
	\begin{align*}
		\mathcal{E}[BS_{\omega,a}] := \bigcup_{0<\varphi<\omega} \mathcal{E}(BS_{\varphi,a}), \quad \mathcal{M}[BS_{\omega,a}] := \bigcup_{0<\varphi<\omega} \mathcal{M}(BS_{\varphi,a}), \quad
		\mathcal{M}[BS_{\omega,a}]_A := \bigcup_{0<\varphi<\omega} \mathcal{M}(BS_{\varphi,a})_A.
	\end{align*}
	If $f\in  \MOmega_A$, that is, if $f$ is regularizable by $\EOmega$, then we say that ``$f(A)$ is defined by the natural functional calculus {\bf(NFC)} for bisectorial operators''.
	
	\subsection{Extensions of the natural functional calculus (NFC)}\label{extensionCalculusSubsection}
	
	Let $d \in \{-a,a\}$ with $a>0$ and consider an operator $A\in \BSect(\omega,a)$ for which $dI-A$ is invertible, that is, $d \notin \sigma(A)$. Fix $0<\varphi<\omega$. Then, there is an $\varepsilon \in (0,a)$ such that the ball $\overline{B_\varepsilon(d)} \subset \rho (A)$, where $B_\varepsilon (d) := \{ z \in \CC: \, |z-d| < \varepsilon\}$.  Consider the algebra
	\begin{align*}
		&\mathcal{E}_0 (BS_{\varphi,a,d,\varepsilon}) :=\Bigg\{ f\in H^\infty (BS_{\varphi,a}  \backslash \overline{B_\varepsilon(d)}): \, \lim_{z\to -d,\infty} f(z) = 0 \text{ and} 
		\\ & \qquad  \qquad \qquad \qquad\int_{\partial (BS_{\omega',a}  \backslash\overline{B_\varepsilon(d)})} \frac{|f(z)| }{|d-z|} \, |dz| < \infty \quad \text{ for all } \varphi < \omega' \leq \frac{\pi}{2} \Bigg\},
	\end{align*}
	and set $\mathcal{E} (BS_{\varphi,a,d, \varepsilon}) := \mathcal{E}_0 (BS_{\varphi,a,d,\varepsilon}) \oplus \CC(1/(b-z))$ for any $b>a$. One can extend our elementary functional calculus $\Phi$ from $\mathcal{E}(BS_{\varphi,a})$ to the algebra $\mathcal{E} (BS_{\varphi,a,d,\varepsilon})$ by integrating on a positively oriented parametrization of the boundary $\partial (BS_{\omega',a}  \backslash\overline{B_\varepsilon(d)})$ for any $ \varphi < \omega' \leq \frac{\pi}{2}$, where we avoid integration near the point $d$ and therefore the functions in this new algebra do not need to satisfy any regularity conditions near $d$. This yields a new abstract functional calculus $(\mathcal{E}(BS_{\varphi,a,d,\varepsilon}), \mathcal{M}(BS_{\varphi,a}), \Phi)$ which is an extension of the former. 
	
	Since the new algebra $\mathcal{E} (BS_{\varphi,a,d,\varepsilon})$ is larger than $\mathcal{E}(BS_{\varphi,a})$, more functions $f \in \mathcal{M}(BS_{\varphi,a})$ become regularizable.  If $f$ is regularizable by $\mathcal{E}(BS_{\varphi,a,d,\varepsilon})$, we say that {\bf $f(A)$ is defined by the NFC for $d-$invertible operators}. 
	
	A similar extension is possible when $A\in \BSect(\omega,a)$ is bounded. Again, let $0< \varphi < \omega$ and let $R > \max\{a, r(A)\}$, where $r(A)$ denotes the spectral radius of $A$. Then one considers the algebra
	\begin{align*}
		&\Bigg\{f \in H^\infty (BS_{\varphi,a} \cap B_R(0)):  \, \lim_{z\to -a} f(z)= \lim_{z\to a} f(z)= 0 \mbox{ and}\\
		& \qquad \int_{\partial BS_{\omega',a}, |z|<R/2} \frac{|f(z)|}{\min\{|z-a|,|z+a|\}} |dz| < \infty \text{ for all } \varphi < \omega' \leq \frac{\pi}{2} \Bigg\},
	\end{align*}
	and adds the spaces $\CC (1/(b-z))$ and $\CC (1/(b+z))$ for any $b>a$. The elementary calculus of this algebra is constructed along the lines of the former elementary calculus. If a function $f$ is regularizable by this algebra,  {\bf we say that $f(A)$ is defined by the NFC for bounded bisectorial operators}.
	
	Similar remarks apply when $A \in \BSect(\omega,a)$ satisfies any combination of the above properties, avoiding integration near the appropriate selection of the points in $\{-a,a,\infty\}$. A minor difference appears if $a = 0$ and $A$ is invertible: in that case the two top and bottom branches merge without further consequences. Clearly, Lemma \ref{functionalCalculusProperties} remains valid when it is adapted appropriately to any of this NFC.

	\begin{definition}
		Let $A\in \BSect(\omega,a)$ and set $M_A := \{-a,a,\infty\} \cap \widetilde{\sigma}(A)$. From now on, we will denote by $\mathcal{M}[\Omega_A]$ the class of meromorphic functions defined on an open set related to the larger primary functional calculus of $A$. Moreover, $\mathcal{M}_A$ will denote the class of regularizable functions $\mathcal{E}_A$ of the larger natural functional calculus of $A$. For instance, if $M_A = \{-a, \infty\}$, then $\mathcal{M}[\Omega_A]$ and $\mathcal{M}_A$ will refer to the calculus of $a-$invertible bisectorial operators; if $M_A = \{-a\}$, then they will refer to the calculus of $a-$invertible and bounded bisectorial operators, and so on.
		
Finally, if $f\in  \mathcal{M}_A$, that is, if $f$ is regularizable by $\mathcal{E}_A$ for the appropriate primary functional calculus of $A$, then we say that $f(A)$ is defined by the natural functional calculus (NFC) of $A$.
	\end{definition}
	
Next, our goal is to give a sufficient condition for a meromorphic function $f \in \mathcal{M}(BS_{\varphi,a})$ to be defined by the NFC of $A$. We will ask $f$ to satisfy a regularity property near the singular points $M_A:=\{-a,a,\infty\} \cap \widetilde{\sigma}(A)$.
	
	\begin{definition}
		Let $A \in \BSect(\omega,a)$, $f\in \mathcal{M}[\Omega_A]$ and $d \in M_A\cap\{-a,a\}$. 
		\begin{enumerate}[(a)]	
			\item 	We say that $f$ is regular at $d$ if $\lim_{z \to d} f(z) =: c_d \in \CC$ exists and, for small enough $\varepsilon > 0$ and some $\varphi< \omega$,
			\begin{align*}
				\int_{\partial (BS_{\omega',a} \cap B_\varepsilon(d))}  \left|\frac{f(z)-c_d}{z-d}\right| |dz| < \infty, \quad \text{for all } \omega' \in \left(\varphi, \frac{\pi}{2}\right).
			\end{align*}
			\item 	Similarly, if $\infty \in M_A$, we say that $f$ is regular at $\infty$ if $\lim_{z\to \infty} f(z) =: c_\infty \in \CC$ exists and, for large enough $R>0$,  and some $\varphi< \omega$,
			\begin{align*}
				\int_{\partial BS_{\omega',a}, |z| > R}  \left|\frac{f(z)-c}{z}\right| |dz| < \infty, \quad \text{for all } \omega' \in \left(\varphi, \frac{\pi}{2}\right).
			\end{align*}
			\item 	We say that $f$ is quasi-regular at $d\in M_A$ if $f$ or $1/f$ is regular at $d$. 
			
			\item Finally, we say that $f$ is (quasi)-regular at $M_A$ if $f$ is (quasi)-regular at each point of $M_A$.
		\end{enumerate}
	\end{definition}
	
\begin{remark}
	Note that if $f$ is regular at $M_A$ with every limit being not equal to $0$, then $1/f$ is also regular at $M_A$. If $f$ is quasi-regular at $M_A$, then $\mu-f$ and $1/f$ are also quasi-regular at $M_A$ for each $\mu \in \CC$. A function $f$ which is quasi-regular at $M_A$ has well defined limits in $\CC_\infty$ as $z$ tends to each point of $M_A$.
\end{remark}

	\begin{lemma}\label{differentNFCLemma}
		Let $A\in \BSect(\omega,a)$ and $f\in \mathcal{M}[\Omega_A]$. Assume that $f$ is regular at $M_A$ and that all the poles of $f$ are contained in $\CC \backslash \sigma_p(A)$. Then, $f(A)$ is defined by the NFC of $A$, that is, $f \in \mathcal{M}_A$. More precisely, the following assertions hold true.
		\begin{enumerate}[(a)]
			\item If $M_A = \{-a,a,\infty\}$,  then $f(A)$ is defined by the NFC for bisectorial operators.
			\item If $M_A = \{-a, \infty\}$,  then $f(A)$ is defined by the NFC for $a$-invertible bisectorial operators. Analogous statement is true if $M_A = \{a,\infty\}$.
			\item If $M_A =  \{-a,a\}$, then $f(A)$ is defined by the NFC for bounded bisectorial operators.
			\item If $M_A = \{\infty\}$, then $f(A)$ is defined by the NFC for $a$-invertible and $-a$-invertible bisectorial operators.
			\item If $M_A = \{-a\}$, then $f(A)$ is defined by the NFC for $a$-invertible and bounded bisectorial operators. Analogous statement is true if $M_A = \{a\}$.
			\item If $M_A = \emptyset$, then $f(A)$ is defined by the NFC for bounded, $a$-invertible and $-a$-invertible bisectorial operators.
		\end{enumerate}
		Moreover, if the poles of $f$ are contained in $\rho(A)$, then $f(A) \in \mathcal{L}(X)$.
	\end{lemma}
	
	\begin{proof}
		The proof is analogous to the corresponding result for sectorial operators, see \cite[Lemma 6.2]{haase2005spectral}. We omit the details for the sake of brevity.
	\end{proof}

	We conclude this section by giving a spectral inclusion result for bisectorial operators.
	
	\begin{proposition}\label{spectralInclusionProposition}
		Let $A \in \BSect(\omega,a)$, and take $f \in \mathcal{M}_A$ to be quasi-regular at $M_A$. Then 
		\begin{align*}
			\widetilde{\sigma}(f(A)) \subset f(\widetilde{\sigma}(A)).
		\end{align*}
	\end{proposition}
	
	\begin{proof}
		The proof follows as the case of sectorial operators, see \cite[Proposition 6.3]{haase2005spectral}.  We omit the details for the sake of brevity.
	\end{proof}

	\section{From bisectorial to sectorial operators and the NFC for sectional operators}\label{sectorialSection} 
	
	We start with relationship between bisectorial and sectorial operators.
	
	\subsection{From bisectorial to sectorial operators}
	In this section, we present a connection between bisectorial operators and sectorial operators, namely Theorem \ref{sectorialTheorem}. The method to prove this connection is based on the proof of the scaling property for sectorial operators given in \cite[Proposition 5.2]{auscher1997holomorphic}, but its proof requires longer and more sophisticated techniques because of the more intricate setting.

	We start with a new definition referring to the limit behavior of a function at the singular points $\{-a,a,\infty\}$. By $|f(z)| \sim |g(z)|$ as $z \to d \in \CC_\infty$, we mean that there is some neighborhood $\Omega$ of $d$ such that $|g(z)| \lesssim |f(z)| \lesssim |g(z)|$ for all $z \in \Omega$.  
	
	\begin{definition}\label{polynomialLimits}
		Let $A\in \BSect(\omega,a)$, $f\in \mathcal{M}[\Omega_A]$, $d \in \overline{\mathcal{D}(f)}$, and $c \in \CC$. 
		\begin{enumerate}[(a)]
			\item We say that $f(z) \to c$ exactly polynomially as $z\to d$ if there is $\alpha > 0$ such that $|f(z) - c|  \sim |z-d|^\alpha$ as $z \to d$ if $d \in \CC$, or such that $|f(z)-c| \sim |z|^{-\alpha}$ as $z \to \infty$ if $d = \infty$.
			
			\item We say that $f(z) \to \infty$ exactly polynomially as $z\to d$ if $(1/f)(z) \to 0$ exactly polynomially as $z\to d$.
			
		\end{enumerate}
	\end{definition}
	
	Now,  we fix some notations for the rest of this section. From now on, $A$ will be a bisectorial operator on a Banach space $X$ of angle $\omega \in (0, \pi/2]$ and half-width $a \geq 0$, i.e. $A \in \BSect(\omega,a)$, and recall that $M_A = \{a,-a,\infty\} \cap \widetilde{\sigma}(A)$. For any $\lambda \in \CC, \, f \in \mathcal{M}[\Omega_A]$,  we let $R_f^\lambda \in \mathcal{M}[\Omega_A]$ be given by
	\begin{align}\label{hf}
		R_f^\lambda (z) := \frac{\lambda}{\lambda-f(z)}, \qquad z \in \mathcal{D}(f).
	\end{align}
	Moreover, we will also consider $\auxangle \in [0, \pi)$ and a function $\auxfunct \in \mathcal{M}_A$ satisfying the following conditions:
	\begin{enumerate}[(a)]
		\item $\mathcal{R}(\auxfunct) \subset \overline{S_\auxangle}\cup \{\infty\}$.
		\item $\auxfunct$ is quasi-regular at $M_A$. In particular, it has limits in $M_A$, which we denote by $c_d \in \CC_\infty$ for $d \in M_A$.
		\item $\auxfunct$ has exactly polynomial limits at $M_A \cap \auxfunct^{-1}(\{0,\infty\})$.	
	\end{enumerate}
	
	\begin{remark}\label{openMappingRemark}
		By the open mapping theorem, Property (a) implies that $\auxfunct$ does not have any zeros (unless $\auxfunct =0$) or poles in $\mathcal{D}(\auxfunct)$. In particular, if $\auxfunct\neq 0$, then both $\auxfunct$ and $\auxfunct^{-1}$ are holomorphic.
	\end{remark}
	
	Next,  we present a family of functions which will be crucial to prove our main result of this section. Recall that a meromorphic function $f$ belongs to $H(A) \subset \mathcal{M}_A$ if and only if $f(A)$ is a bounded operator.
	
	\begin{definition}\label{boundingFunctionsDef}
		Assume that we have a family of functions $(f^\lambda)_{\lambda \notin \overline{S_\auxangle}} \subset H(A)$. We say that $(f^\lambda)_{\lambda \notin \overline{S_\auxangle}}$ makes $(R_\auxfunct^\lambda)_{\lambda \notin \overline{S_{\auxangle}}}$ $\varepsilon-$uniformly bounded at $d \in M_A$ with respect to the NFC of $A$ if, for any $\varepsilon \in (0,\pi - \auxangle)$, it satisfies the following properties:
		\begin{enumerate}[(a)]
			\item $|f^\lambda(z)|$ is uniformly bounded for all $z \in \mathcal{D}(f^\lambda)$ and $\lambda \notin \overline{S_{\auxangle+\varepsilon}}$.
			\item $\|f^\lambda(A)\|_{\mathcal L(X)}$ is uniformly bounded for all $\lambda \notin \overline{S_{\auxangle+\varepsilon}}$.
			\item Let $\Gamma$ be an integration path for the NFC of $A$ (see Section \ref{extensionCalculusSubsection}). Then, there exists a neighborhood $\Omega_{d'}$ containing each $d' \in M_A \backslash\{d\}$ for which
			\begin{align*}
				\sup_{\lambda \notin \overline{S_{\auxangle+\varepsilon}}} \int_{\Gamma \cap \Omega_{d'}} |f^\lambda(z)|\|R(z,A)\|_{\mathcal L(X)} \, |dz| < \infty \quad \text{for each } d' \in M\backslash\{d\}.
			\end{align*}
			\item Let $\Gamma$ be as above. Then, there exists a neighborhood $\Omega_d$ containing $d$ for which
			\begin{align*}
				\sup_{\lambda \notin \overline{S_{\auxangle+\varepsilon}}} \int_{\Gamma \cap \Omega_d} |R_\auxfunct^\lambda(z) - f^\lambda(z)| \|R(z,A)\|_{\mathcal L(X)}\, |dz| < \infty.
			\end{align*}
		\end{enumerate}
	\end{definition}
	
	The following lemmas will be useful in finding the family of functions that makes $(R_\auxfunct^\lambda)_{\lambda \notin \overline{S_{\auxangle}}}$ $\varepsilon-$uniformly bounded. Let $d(z,\Omega)$ denote the distance between a point $z \in \CC$ and a set $\Omega \subset \CC$.
	
	\begin{lemma}\label{basicIneqLemma}
		Let $\varepsilon > 0$. We have that $d(z,\CC\backslash \overline{S_{\auxangle+\varepsilon}}) \gtrsim |z|$ for all $z \in \overline{S_\auxangle}$ and that $d(w, \overline{S_\auxangle}) \gtrsim |w|$ for all $w \notin \overline{S_{\auxangle+\varepsilon}}$. As a consequence, $|w/(w-z)|$ and $|z/(w-z)|$ are uniformly bounded for all $z \in \overline{S_\auxangle}$ and $w \notin \overline{S_{\auxangle+\varepsilon}}$.
	\end{lemma}
	
	\begin{proof}
		The first two inequalities follow from the fact that $d(z, \CC\backslash \overline{S_{\auxangle+\varepsilon}}) = |z| \sin(\auxangle + \varepsilon - |\arg(z)|) \geq |z| \sin \varepsilon$ and $d(w, \overline{S_{\auxangle}}) = |w| \sin(|\arg w|-\auxangle) \geq |w| \sin \varepsilon$. The other inequalities follow from what we have already proven,  and that $|z-w| \geq \max\{ d(z, \CC\backslash \overline{S_{\auxangle+\varepsilon}}) , d(w, \overline{S_{\auxangle}})\}$ for all $z \in \overline{S_\auxangle}$ and $w \notin \overline{S_{\auxangle+\varepsilon}}$.
	\end{proof}
	
	Lemma \ref{basicIneqLemma} will be used very frequently throughout the paper. We will not make no further reference to it for the sake of brevity.  In particular, as an immediate application we have the following result.
	
	\begin{lemma}\label{IneqLemma}
		Let $c \in \overline{S_{\auxangle}}\backslash \{0\}$, $\varepsilon \in (0, \pi-\auxangle)$ and $f \in \mathcal{M}[\Omega_A]$ such that $\mathcal{R}(f) \subset \overline{S_\auxangle} \cup \{\infty\}$.  Then,  
		\begin{align*}
			\left|R_f^\lambda(z) - \frac{\lambda}{\lambda-c}\right| &\lesssim \min\{1,|f(z) - c|\},
			\\ \left|R_f^\lambda(z) \right| &\lesssim \min\{1,|\lambda| |f(z)|^{-1}\},
			\\ \left|R_f^\lambda(z) - 1 \right| &\lesssim \min\{1,|\lambda|^{-1} |f(z)|\},
		\end{align*}
		where all inequalities hold for all $z \in \mathcal{D}(f)$ and $\lambda \notin \overline{S_{\auxangle+\varepsilon}}$.
	\end{lemma}
	
	\begin{proof}
		It follows from Lemma \ref{basicIneqLemma} that all the above functions are uniformly bounded.
		
		Now, let $c \in \overline{S_{\auxangle}}$.  Applying Lemma \ref{basicIneqLemma}, one gets that
		\begin{align*}
			\left| R_f^\lambda (z) - \frac{\lambda}{\lambda-c}\right| & = \left|\frac{\lambda}{\lambda - f(z)} \frac{f(z) - c}{\lambda-c}\right| \lesssim \left|\frac{f(z)-c}{\lambda-c}\right| \lesssim |f(z) - c|,
		\end{align*}
		for all $z \in \mathcal{D}(f), \, \lambda \notin \overline{S_{\auxangle+\varepsilon}}$. Likewise, one obtains that
		\begin{align*}
			|R_f^\lambda(z)| = \left|\frac{\lambda}{\lambda - f(z)}\right| \lesssim |\lambda| |f(z)|^{-1},
			\qquad |R_f^\lambda(z)-1| = \left|\frac{f(z)}{\lambda - f(z)}\right| \lesssim |\lambda|^{-1} |f(z)|,
		\end{align*}
		for all $z \in \mathcal{D}(f), \, \lambda \notin \overline{S_{\auxangle+\varepsilon}}$. The proof is finished.
	\end{proof}
	
The following lemma gives an integral bound which will be very useful to prove Proposition \ref{existenceProposition} below.

	\begin{lemma}\label{integralBoundLemma}
		Let $I \subset \RR^+$ be a measurable subset and let $(f_\nu)_{\nu}:I \to \CC$ be a family of complex-valued functions. Let $F_1, F_2:I \to \RR^+$ be some positive functions which are integrable with respect to the measure $dx/x$, and let $r_\nu>0$ for all indices $\nu$, and $s_n, t_m> 0$ for $n = 1, ..., N$, $m = 1,..., M$ for some $N,M\in \NN$. Assume that 
		\begin{align}\label{f-gamma}
			|f_\nu(x)| \lesssim \min\left\{F_1(x) + \sum_{n\leq N} (r_\nu x)^{s_n}, F_2(x) + \sum_{m\leq M} (r_\nu x)^{-t_m}\right\}, 
		\end{align}
		for all $x \in I$ and all indices $\nu$. Then
		\begin{align*}
			\sup_{\nu} \int_I |f_\nu(x)| \frac{dx}{x} < \infty.
		\end{align*}
	\end{lemma}
	
	\begin{proof}
		By adding terms of the type $(r_\nu x)^{t_m}$ to the first expression inside the minimum in \eqref{f-gamma},  and terms of the type $(r_\nu x)^{-s_n}$, one can assume that $N = M$ and $s_n = t_n$ for all $n = 1,...,N$. It follows that
		\begin{align*}
			\int_I |f_\nu(x)| \frac{dx}{x} & \lesssim \int_I \Big(F_1(x) + F_2(x)\Big) \frac{dx}{x} 
			+ \int_I \min\left\{\sum_{n\leq N} (r_\nu x)^{s_n}, \sum_{n\leq N} (r_\nu x)^{-s_n}\right\} \frac{dx}{x}.
		\end{align*}
		By the integrability condition on $F_1, F_2$, it suffices to bound the second term for all indices $\nu$. By bounding the integral on $I$ by the integral on $(0,\infty)$, and a simple change of variable, one gets that
		\begin{align*}
			&\int_I \min\left\{\sum_{n\leq N} (r_\nu x)^{s_n}, \sum_{n\leq N} (r_\nu x)^{-s_n}\right\} \frac{dx}{x} 
			= \int_0^\infty \min\left\{\sum_{n\leq N} x^{s_n}, \sum_{n\leq N} x^{-s_n}\right\} \frac{dx}{x} 
			\\ & \quad =  \sum_{n\leq N} \left(\int_0^1 x^{s_n-1} \, dx + \int_1^\infty x^{-s_n-1} \, dx \right)< \infty.
		\end{align*}
		The proof is concluded.
	\end{proof}
	
	In order to prove the main result of this section, some integrals related to resolvent operators need to be bounded. The techniques to bound them vary from one case to another, as shows the proposition below.
	
	\begin{proposition}\label{existenceProposition}
		For each point $d \in M$, there exists a family of functions $(f^\lambda)_{\lambda \notin \overline{S_{\auxangle}}}$ that makes $(R_\auxfunct^\lambda)_{\lambda \notin \overline{S_{\auxangle}}}$ $\varepsilon-$uniformly bounded at $d$ with respect to the NFC of $A$.
	\end{proposition}
	
	\begin{proof}
		We will proceed by examining all the possible cases. Throughout the proof $\varepsilon$ is any appropriate number in $ (0,\pi - \auxangle)$ whenever it appears. Also,  $b > a$ for the rest of the proof. We proceed in several steps.
		
		{\bf Step 1}: Let $d = a$ and $c_a \in \overline{S_{\auxangle}} \backslash \{0,\infty\}$.  We claim that the family of functions given by
		\begin{align*}
			f_{a, c_a}^\lambda(z) := \frac{\lambda}{\lambda-c_a} \frac{b^2-a^2}{2a} \frac{a+z}{b^2-z^2}, \quad z  \in \mathcal{D}(\auxfunct),
		\end{align*}
		makes $(R_\auxfunct^\lambda)_{\lambda \notin \overline{S_{\auxangle}}}$ $\varepsilon-$uniformly bounded  at $a$.  Indeed,  it follows from Lemma \ref{basicIneqLemma} that $|f_{a,c_a}^\lambda (z)|$ is uniformly bounded for all $z \in \mathcal{D}(f_{a,c_a}^\lambda) = \mathcal{D}(\auxfunct)$ and $\lambda \notin \overline{S_{\auxangle+\varepsilon}}$. Moreover, $f_{a,c_a}^\lambda \in H(A)$ with 
		\begin{align*}
			f_{a,c_a}^\lambda (A) = \frac{\lambda}{\lambda-c_a} \frac{b^2-a^2}{2a} (A+aI) R(b^2,A^2),
		\end{align*}
		so $\|f_{a,c_a}^\lambda (A)\|_{\mathcal L(X)}$ is also uniformly bounded for all $\lambda \notin \overline{S_{\auxangle+\varepsilon}}$ (recall that $\sigma(A^2)= \sigma(A)^2$). It is clear that the integrability property (c) in Definition \ref{boundingFunctionsDef} holds,  since again $\lambda/(\lambda-c_a)$ is bounded and $(z+a)/(b^2-z^2)$ is integrable with respect to $\|R(z,A)\|_{\mathcal L(X)}|dz|$ at the neighborhoods of $-a$ and $\infty$. Finally, we have that
		\begin{align*}
			\left|R_\auxfunct^\lambda(z) - f_{a,c_a}^\lambda(z)\right| &\leq \left| R_\auxfunct^\lambda(z) - \frac{\lambda}{\lambda - c_a}\right| + \left|\frac{\lambda}{\lambda-c_a}\right| \left|\frac{b^2-a^2}{2a} \frac{z+a}{b^2-z^2} - 1\right|
			\lesssim |\auxfunct(z) - c_a| + |z-a|,
		\end{align*}
		where the first estimate is obtained by an application of Lemma \ref{IneqLemma},  and the second one by using Lemma \ref{basicIneqLemma} and Taylor's expansion of order $1$. Since $\auxfunct$ is regular at $a$ with limit $c_a$, it follows that $|R_\auxfunct^\lambda(z) - f_{a,c_a}^\lambda(z)|$ satisfies the integrability property (d) in Definition \ref{boundingFunctionsDef} and the claim is proven.
		
		{\bf Step 2}:  Next,  let $d = a$ and $c_a = 0$. Since $a \in M_A \cap \auxfunct^{-1}(\{0,\infty\})$,  it follows that in a neighborhood $\Omega_a$ containing $a$ and a real number $\alpha >0$, we have that $|\auxfunct(z)| \sim |z-a|^\alpha$ for all $z \in \Omega_a \cap \mathcal{D}(\auxfunct)$. We consider the family of functions given by
		\begin{align*}
			f_{a,0}^\lambda (z) := \frac{|\lambda|^{1/\alpha}}{|\lambda|^{1/\alpha}+a-z} \frac{b^2-a^2}{2a} \frac{a+z}{b^2-z^2}, \quad z\in \mathcal{D}(\auxfunct).
		\end{align*}
		Let us show that $(f_{a,0}^\lambda)_{\lambda \notin \overline{S_{\auxangle+\varepsilon}}}$ satisfies the desired properties. First of all, it is clear that $|f_{a,0}^\lambda(z)|$ is uniformly bounded for all $z \in \mathcal{D}(f_{a,0}^\lambda) = \mathcal{D}(\auxfunct)$ and $\lambda \notin \overline{S_{\auxangle+\varepsilon}}$. Moreover, $f_{a,0}^\lambda \in H(A)$ with 
		\begin{align*}
			f_{a,0}^\lambda (A) = \frac{b^2-a^2}{2a} (A+aI) R(b^2, A^2) |\lambda|^{1/\alpha} R(|\lambda|^{1/\alpha}, A-aI).
		\end{align*}
		Thus, it follows from the definition of bisectorial operators that $\|f_{a,0}^\lambda (A)\|_{\mathcal L(X)}$ is uniformly bounded for all $\lambda \notin \overline{S_{\auxangle+\varepsilon}}$. It is readily seen that it satisfies Property (c) in Definition \ref{boundingFunctionsDef},  since $|\lambda|^{1/\alpha}/(|\lambda|^{1/\alpha}+a-z)$ is uniformly bounded. 
		
		Next, we consider $|R_\auxfunct^\lambda (z) - f_{c_a,0}^\lambda (z)|$ in $\Omega_a \cap \mathcal{D}(\auxfunct)$. On the one hand, by the triangle inequality and various applications of Lemmas \ref{basicIneqLemma} and \ref{IneqLemma},  we get 
		\begin{align*}
			| R_\auxfunct^\lambda (z) - f_{a,0}^\lambda (z)| & \leq \left|R_\auxfunct^\lambda(z)\right| + \frac{b^2-a^2}{2a}\left|\frac{a+z}{b^2-z^2}\right| \left|\frac{|\lambda|^{1/\alpha}}{|\lambda|^{1/\alpha}+a-z}\right|
			\\ & \lesssim |\lambda| |\auxfunct(z)|^{-1} + |\lambda|^{1/\alpha} |z-a|^{-1}
			\lesssim |\lambda^{-1/\alpha} (z-a)|^{-\alpha} + |\lambda^{-1/\alpha} (z-a)|^{-1},
		\end{align*} 
		for all $z \in \Omega_a \cap \mathcal{D}(\auxfunct)$ and $\lambda \notin \overline{S_{\auxangle+\varepsilon}}$. On the other hand, it follows that
		\begin{align*}
			| R_\auxfunct^\lambda (z) - f_{a,0}^\lambda (z)| & \leq |R_\auxfunct^\lambda (z) -1| + |f_{a,0}^\lambda(z) -1|. 
		\end{align*}
		Another application of Lemma \ref{IneqLemma} yields that $|R_\auxfunct^\lambda (z) -1| \lesssim |\lambda|^{-1} |\auxfunct(z)| \lesssim |\lambda^{-1/\alpha} (z-a)|^\alpha$ for all $z \in \Omega_a \cap \mathcal{D}(\auxfunct)$ and $\lambda \notin \overline{S_{\auxangle+\varepsilon}}$. Moreover, one gets that
		\begin{align*}
			|f_{a,0}^\lambda(z) -1| &\leq \left|f_{a,0}^\lambda (z) - \frac{b^2-a^2}{2a} \frac{a+z}{b^2-z^2}\right| + \left|\frac{b^2-a^2}{2a} \frac{a+z}{b^2-z^2} -1 \right|
			\\ & \lesssim \left|\frac{a-z}{|\lambda|^{1/\alpha}+a-z}\right| + |z-a|
			\lesssim |\lambda^{-1/\alpha} (z-a)| + |z-a|,
		\end{align*}
		for all $z \in \Omega_a \cap \mathcal{D}(\auxfunct)$ and $\lambda \notin \overline{S_{\auxangle+\varepsilon}}$.  Summarizing,  if we set $U_\lambda(z) := |\lambda|^{1/\alpha} |z-a|$,  then we obtain that
		\begin{align*}
			|R_\auxfunct^\lambda(z) - f_{a,0}^\lambda(z)| \lesssim \min\left\{ \sum_{j\in \{1,\alpha\}} U_\lambda(z)^{-j}, |z-a| + \sum_{j\in \{1,\alpha\}} U_\lambda(z)^j   \right\},
		\end{align*}
		for all $z \in \Omega_a \cap \mathcal{D}(\auxfunct)$ and $\lambda \notin \overline{S_{\auxangle+\varepsilon}}$.  An application of Lemma \ref{integralBoundLemma} together with the bound of the resolvent of a bisectorial operator, yield that $(f_{a,0}^\lambda)_{\lambda \notin \overline{S_{\auxangle}}}$ satisfies Property (d) in Definition \ref{boundingFunctionsDef}, so in fact $(f_{a,0}^\lambda)_{\lambda \notin \overline{S_{\auxangle}}}$ makes $(R_\auxfunct^\lambda)_{\lambda \notin \overline{S_{\auxangle}}}$ $\varepsilon-$uniformly bounded at $a$ with respect to the NFC of $A$.
		
		{\bf Step 3}:  Next,  let $d = a$ and $c_a = \infty$.  By hypothesis,  in a neighborhood $\Omega_a$ containing $a$,  and a real number $\alpha >0$, we have that $|\auxfunct(z)| \sim |z-a|^{-\alpha}$ for all $z \in \Omega_a \cap \mathcal{D}(\auxfunct)$. Set
		\begin{align*}
			f_{a,\infty}^\lambda (z) :=  \frac{a-z}{|\lambda|^{-1/\alpha}+a-z} \frac{b^2-a^2}{2a}\frac{a+z}{b^2-z^2}, \quad z \in \mathcal{D}(\auxfunct).
		\end{align*}
		Similar reasoning as in the above cases together with the observation that 
		\begin{align*}
			(a I -A)R(|\lambda|^{-1/\alpha},A-aI) = I - |\lambda|^{-1/\alpha} R(|\lambda|^{-1/\alpha}, A-aI),
		\end{align*}
		leads to the fact that the family $(f_{a,\infty}^\lambda)_{\lambda \notin \overline{S_{\auxangle}}}$ satisfies Properties (a), (b) and (c) in Definition \ref{boundingFunctionsDef}.  By Lemma \ref{basicIneqLemma},  it easily follows that $|f_{a,\infty}^\lambda(z)| \lesssim |\lambda|^{1/\alpha} |z-a|$. Therefore,  the triangle inequality and  an application of Lemma \ref{IneqLemma} yield
		\begin{align*}
			|R_\auxfunct^\lambda(z) - f_{a,\infty}^\lambda (z)| & \leq |\lambda^{1/\alpha} (z-a)|^\alpha + |\lambda^{1/\alpha}(z-a)|,
		\end{align*}
		for all $z \in \Omega_a \cap \mathcal{D}(\auxfunct)$ and $\lambda \notin \overline{S_{\auxangle+\varepsilon}}$.  This implies that
		\begin{align*}
			|R_\auxfunct^\lambda(z) - f_{a,\infty}^\lambda (z)| &\leq \left|R_\auxfunct^\lambda (z) - \frac{a-z}{|\lambda|^{-1/\alpha}+a-z} \right| 
			+ \left|\frac{a-z}{|\lambda|^{-1/\alpha}+a-z}\right| \left|\frac{b^2-a^2}{2a}\frac{a+z}{b^2-z^2} - 1\right| 
			\\ &\lesssim \left|R_\auxfunct^\lambda (z) - \frac{a-z}{|\lambda|^{-1/\alpha}+a-z} \right|  + |z-a|,
		\end{align*}
		where we have used again the Taylor expansion of order $1$ and the fact that $|(a-z)/(|\lambda|^{-1/\alpha}+a-z)|$ is uniformly bounded. In addition, it follows that
		\begin{align*}
			\left|R_\auxfunct^\lambda (z) - \frac{a-z}{|\lambda|^{-1/\alpha}+a-z} \right|  =& \left|\frac{\lambda |\lambda|^{-1/\alpha} + \auxfunct(z) (a-z)}{(\lambda - \auxfunct(z))(|\lambda|^{-1/\alpha}+a-z)}\right|
			\\  \leq & \left|\frac{\lambda |\lambda|^{-1/\alpha}}{(\lambda - \auxfunct(z))(|\lambda|^{-1/\alpha}+a-z)}\right| + \left|\frac{ \auxfunct(z) (a-z)}{(\lambda - \auxfunct(z))(|\lambda|^{-1/\alpha}+a-z)}\right|
			\\ \lesssim & |\lambda^{1/\alpha}(z-a)|^{-1} + |\lambda^{1/\alpha} (z-a)|^{-\alpha}, 
		\end{align*}
		for all $z \in \Omega_a \cap \mathcal{D}(\auxfunct)$ and $\lambda \notin \overline{S_{\auxangle+\varepsilon}}$, 
		where we have used various applications of Lemmas \ref{basicIneqLemma} and \ref{IneqLemma} in the last step. Finally, reasoning as in the case before with Lemma \ref{integralBoundLemma}, one obtains that $(f_{a,\infty}^\lambda)_{\lambda \notin \overline{S_{\auxangle}}}$ satisfies also Property (d) in Definition \ref{boundingFunctionsDef}.
		
		{\bf Step 4}:  Similar reasoning as in Step 1 shows that, if either $-a, \infty \in M_A$ with $c_{-a}, c_\infty \in \overline{S_{\auxangle}} \backslash \{0,\infty\}$, the families of functions given by
		\begin{align*}
			f_{-a,c_{-a}}^\lambda(z) &:= \frac{\lambda}{\lambda-c_{-a}} \frac{b^2-a^2}{2a} \frac{a-z}{b^2-z^2}, \quad z \in \mathcal{D}(\auxfunct)
			\\ f_{\infty,c_\infty}^\lambda(z) &:= \frac{\lambda}{\lambda-c_\infty} \frac{a^2-z^2}{b^2-z^2}, \quad \quad\quad \qquad z \in \mathcal{D}(\auxfunct),
		\end{align*}
		make $(R_\auxfunct^\lambda)_{\lambda \notin \overline{S_{\auxangle}}}$ $\varepsilon-$uniformly bounded at $-a$ and $\infty$,  respectively,  with respect to the NFC of $A$.
		
		{\bf Step 5}:  Let either $-a,\infty \in M_A$ with polynomial limits $c_a, c_\infty = 0$ of exactly order $\alpha>0$, that is, $|\auxfunct(z)| \sim |z+a|^\alpha$ near $z=-a$ and $|\auxfunct(z)| \sim |z|^{-\alpha}$ near $z=\infty$,  respectively. Proceeding as in Step 2,  one has that the families of functions given by
		\begin{align*}
			f_{-a,0}^\lambda (z) &:= \frac{|\lambda|^{1/\alpha}}{|\lambda|^{1/\alpha}+a+z} \frac{b^2-a^2}{2a} \frac{a-z}{b^2-z^2}, \quad z \in \mathcal{D}(\auxfunct),
			\\ f_{\infty,0}^\lambda(z) &:= \frac{b-z}{|\lambda|^{-1/\alpha}+b-z} \frac{a^2-z^2}{b^2-z^2}, \quad \qquad \quad z \in \mathcal{D}(\auxfunct),
		\end{align*}
		satisfy the desired properties.
		
		{\bf Step 6}:  Finally, assume that either $-a,\infty \in M_A$ with polynomial limits $c_a, c_\infty = \infty$ of exactly order $\alpha>0$, i.e. $|\auxfunct(z)| \sim |z+a|^{-\alpha}$ near $z=-a$ and $|\auxfunct(z)| \sim |z|^\alpha$ near $z=\infty$,  respectively.  Analogous computations as in Step 3 yields to the fact that the families of functions given by
		\begin{align*}
			f_{-a,\infty}^\lambda (z) &:=  \frac{a+z}{|\lambda|^{-1/\alpha}+a+z} \frac{b^2-a^2}{2a}\frac{a-z}{b^2-z^2}, \quad z \in \mathcal{D}(\auxfunct),
			\\ f_{\infty,\infty}^\lambda(z) &:= \frac{|\lambda|^{1/\alpha}}{|\lambda|^{1/\alpha}+b-z} \frac{a^2-z^2}{b^2-z^2}, \quad \qquad\qquad z \in \mathcal{D}(\auxfunct),
		\end{align*}
		make $(R_\auxfunct^\lambda)_{\lambda \notin \overline{S_{\auxangle}}}$ $\varepsilon$-uniformly bounded at $-a$ and $\infty$ respectively with respect to the NFC of $A$. The proof is complete.
	\end{proof}
	
	\begin{remark}
		If $M_A$ is a proper subset of $\{-a,a,\infty\}$ one can slightly simplify the families of functions given in the proof of Proposition \ref{existenceProposition}. More precisely, for $d\in M_A$, one can eliminate the terms in $f_d^\lambda$ that make the functions uniformly bounded and integrable in the rest of the points in $M_A$. For instance, if $M_A = \{a,\infty\}$ then the behavior of the functions near $-a$ is irrelevant. Thus, if $c_a \notin \{0,\infty\}$, then the family of functions given by
		\begin{align*}
			\frac{\lambda}{\lambda-c_a} \frac{b-a}{b-z}, \quad z \in \mathcal{D}(\auxfunct),
		\end{align*}
		also satisfies the desired properties.
	\end{remark}
	
	\begin{remark}\label{integrationRemark}
		Let $(f_d^\lambda)_{\lambda \notin \overline{S_{\auxangle}}}$ be a family of functions that makes $(R_\auxfunct^\lambda)_{\lambda \notin \overline{S_{\auxangle}}}$ $\varepsilon-$uniformly bound at each $d\in M_A$ with respect to the NFC of $A$. From the bounds appearing in the proof of Proposition \ref{existenceProposition},  one obtains that in fact $(R_\auxfunct^\lambda - \sum_{d\in M_A} f_d^\lambda) \in \mathcal{E}_0$ for all $\lambda \notin \overline{S_{\auxangle}}$.
	\end{remark}

	We are now ready to state the main result of this section.
	
	\begin{theorem}\label{sectorialTheorem}
		Let $(\omega,a) \in (0,\pi/2] \times [0,\infty)$ and $\finalangle \in [0,\pi)$. Let $A \in \BSect(\omega,a)$ in a Banach space $X$ and $\finalfunct \in \mathcal{M}_A$. Assume the following:
		\begin{enumerate}[(a)]
			\item For any $\auxangle > \finalangle$, one can find $\varphi \in (0, \omega)$ such that $\finalfunct(BS_{\varphi,a}) \subset \overline{S_{\auxangle}} \cup\{\infty\}$.
			\item $\finalfunct$ is quasi-regular at $M_A$.
			\item $\finalfunct$ has exactly polynomial limits at $M_A \cap \finalfunct^{-1}(\{0,\infty\})$.
		\end{enumerate}
	Then, $\finalfunct(A)$ is a sectorial operator of angle $\finalangle$.
	\end{theorem}
	
	\begin{proof}
		Our result will follow once we have proven that $\finalfunct(A)$ is a sectorial operator of angle $\auxangle$ for all $\auxangle > \finalangle$. Indeed, if that's true, we will have that
		\begin{align*}
			\finalangle \geq \inf_{\auxangle \in [0,\pi)} \{\auxangle \, : \, \finalfunct(A) \text{ is sectorial of angle } \auxangle\},
		\end{align*}
		which implies that $\finalfunct(A)$ is a sectorial operator of angle $\finalangle$, see the beginning of \cite[Section 2.1]{haase2006functional}.
		
		Indeed,  let $\auxangle > \finalangle$ and set $\auxfunct := \finalfunct|_{BS_{\varphi,a}}$, where $\varphi\in (0,\omega)$ is chosen such that $\mathcal{R}(\auxfunct) \subset \overline{S_\auxangle}\cup \{\infty\}$. Notice that $\auxfunct(A) = \finalfunct(A)$. Now, the spectral inclusion $\widetilde{\sigma}(\auxfunct(A)) \subset \overline{S_\auxangle} \cup \{\infty\}$ holds by Proposition \ref{spectralInclusionProposition}. It remains to prove the bound for the resolvent. More precisely,  we have to show that for all $\varepsilon > 0$, 
		\begin{align*}
			\sup_{\lambda \notin \overline{S_{\auxangle+\varepsilon}}} \| \lambda R(\lambda, \auxfunct(A))\|_{\mathcal L(X)} < \infty.
		\end{align*}
		By Lemma \ref{functionalCalculusProperties}, it follows that $\lambda R(\lambda, \auxfunct(A)) = R_\auxfunct^\lambda (A)$ for all $\lambda \notin \overline{S_\auxangle}$. Also, by Proposition \ref{existenceProposition}, we have that for each $d\in M_A$, there exist some families of functions $(f_d^\lambda)_{\lambda \notin \overline{S_{\auxangle}}}$ which make $(R_\auxfunct^\lambda)_{\lambda \notin \overline{S_{\auxangle}}}$ $\varepsilon$-uniformly bounded at $d$ with respect to the NFC of $A$. We have that
		\begin{align}\label{hrg}
			\|\lambda R(\lambda, \auxfunct(A))\|_{\mathcal L(X)} & \leq \left\|\left(R_\auxfunct^\lambda - \sum_{d\in M_A} f_d^\lambda\right)(A)\right\|_{\mathcal L(X)} + \sum_{d\in M_A} \left\|  f_d^\lambda (A)\right\|_{\mathcal L(X)}.
		\end{align}
		By Property (c) in Definition \ref{boundingFunctionsDef}, one has that $\sup_{\lambda \notin \overline{S_{\auxangle+\varepsilon}}} \|f_d^\lambda (A)\|_{\mathcal L(X)} < \infty$ for each $d \in M_A$.  It remains to uniformly bound the first term.
		
		Let $\Gamma$ be an integration path for the primary functional calculus of $A$, and $(\Omega_d)_{d\in M_A}$ some appropriate open sets for which $d\in \Omega_d$ and the uniform integral bounds of Definition \ref{boundingFunctionsDef} hold for each $(f_d^\lambda)_{\lambda \notin \overline{S_\auxangle}}$. Since $(R_\auxfunct^\lambda - \sum_{d\in M_A} f_d^\lambda) \in \mathcal{E}_0$ (see Remark \ref{integrationRemark}), one has that
		\begin{align}\label{hrg-2}
			\left\|\left(R_\auxfunct^\lambda - \sum_{d\in M_A} f_d^\lambda\right)(A)\right\|_{\mathcal L(X)}  \leq \int_\Gamma \left|R_\auxfunct^\lambda (z) - \sum_{d\in M_A} f_d^\lambda (z) \right| \|R(z,A)\|_{\mathcal L(X)}\, |dz|.
		\end{align}
		Next, we split the integral on $\Gamma$ to the sum of integrals on $\Gamma \cap \Omega_d$ for each $d \in M_A$, and $\Gamma \backslash \left(\cup_{d\in M_A} \Omega_d\right)$. Notice that by Property (b) in Definition \ref{boundingFunctionsDef} and Lemma \ref{IneqLemma}, $|f_d^\lambda(z)|$ and $|R_\auxfunct^\lambda(z)|$ are uniformly bounded. Thus, 
		\begin{align*}
			&\sup_{\lambda \notin \overline{S_{\auxangle+\varepsilon}}} \int_{\Gamma \backslash \left(\cup_{d\in M_A} \Omega_d\right)} \left|R_\auxfunct^\lambda (z) - \sum_{d'\in M_A} f_{d'}^\lambda (z) \right| \|R(z,A)\|_{\mathcal L(X)}\, |dz| 
			\lesssim  \int_{\Gamma \backslash \left(\cup_{d\in M} \Omega_d\right)} \|R(z,A)\|_{\mathcal L(X)}\, |dz| < \infty.
		\end{align*}
		Finally, for each $d \in M_A$, one has that
		\begin{align*}
			&\sup_{\lambda \notin \overline{S_{\auxangle+\varepsilon}}} \int_{\Gamma \cap \Omega_d} \left|R_\auxfunct^\lambda (z) - \sum_{d'\in M_A} f_{d'}^\lambda (z) \right| \|R(z,A)\|_{\mathcal L(X)}\, |dz|
			\\ \leq  & \sup_{\lambda \notin \overline{S_{\auxangle+\varepsilon}}} \sum_{d' \in M_A\backslash\{d\}} \int_{\Gamma \cap \Omega_d} |f_{d'}^\lambda (z)| \|R(z,A)\|_{\mathcal L(X)}\, |dz|
			+ \sup_{\lambda \notin \overline{S_{\auxangle+\varepsilon}}} \int_{\Gamma \cap \Omega_d} \left|R_\auxfunct^\lambda (z) - f_d^\lambda (z) \right| \|R(z,A)\|_{\mathcal L(X)}\, |dz| .
		\end{align*}
		But,  these two supremums of the integrals are finite by Properties (c) and (d) in Definition \ref{boundingFunctionsDef},  respectively. Combining these estimates with \eqref{hrg}-\eqref{hrg-2} we get the resolvent bound, and as a consequence $\auxfunct(A) = \finalfunct(A)$ is a sectorial operator of angle $\auxangle$ for all $\auxangle > \finalangle$. The proof is finished.
	\end{proof}
	
	The following corollaries are immediate consequences of Theorem \ref{sectorialTheorem}.  The first one has been already stated in \cite[Theorem 4.6]{baeumer2009unbounded} for the particular case where $-A$ generates a bounded group.
	
	\begin{corollary}\label{StripCor}
		Let $a\geq 0$,  $A \in \BSect(\pi/2, a)$, and let $\alpha > 0$ with $\alpha$ not an odd number, so $\alpha \in (2n-1, 2n+1)$ for a unique $n\in \NN$. Then, for any $\varepsilon >0$,  there exists  $\rho \geq 0$ such that $\rho  I + (-1)^{n} (A+a I)^\alpha$ is a sectorial operator of angle $\pi\left|\frac{\alpha}{2} - n\right|+ \varepsilon$. Moreover, if $a = 0$ then we can take $\rho = 0$.
	\end{corollary}
	
	\begin{corollary}\label{bisectorialCor}
		Let $0<  \omega \leq \frac{\pi}{2}$ and $a \geq 0$. Let $A \in \BSect(\omega,a)$ in a Banach space $X$ and $\finalfunct \in \mathcal{M}_A$. Assume that there are  $\finalangle\in [\pi/2,\pi)$ and $b\geq0$ such that the following hold:
		\begin{enumerate}[(a)]
			\item For any $\auxangle > \finalangle$, one can find $\varphi \in (0, \omega)$ for which $g(BS_{\varphi,a}) \subset \overline{BS}_{\auxangle,b}$.
			\item $\finalfunct$ is quasi-regular at $M_A$.
			\item $\finalfunct$ has exactly polynomial limits at $M_A \cap \finalfunct^{-1}(\{-b,b,\infty\})$.
		\end{enumerate} 
		Then, $\finalfunct(A)$ is a bisectorial operator of angle $\pi-\finalangle$ and half-width $b$. 
	\end{corollary}
	
	\subsection{The natural functional calculus for sectorial operators}\label{sectorialSubsection}

	Most proofs which we have presented in this text are generic, and as a consequence, the results shown here will hold for functional calculus analogous to that in Sections \ref{Sec2} and \ref{extensionCalculusSubsection}. In particular, it is straightforward to adapt the preceding results to the case of NFC of sectorial operators that can be found in \cite{haase2005spectral}, where the reader can find the definitions of the appropriate versions of the function spaces $\mathcal{E} (S_\varphi), \mathcal{E}_0 (S_\varphi), \mathcal{M}(S_\varphi)_A$,... For instance, one obtains the following version of Theorem \ref{sectorialTheorem} adapted to this NFC.
	
	\begin{theorem}
		Let $0 \leq  \omega < \pi$, $\finalangle\in [0,\pi)$,  $A$  a sectorial operator of angle $\omega$ in a Banach space $X$,  and $\finalfunct \in \mathcal{M}_A$ in the sense of the NFC for sectorial operators. Assume that the following hold:
		\begin{enumerate}[(a)]
			\item For any $\auxangle > \finalangle$, one can find $\varphi \in (\omega, \pi)$ such that $\finalfunct(S_\varphi) \subset \overline{S_{\auxangle}} \cup\{\infty\}$.
			\item $\finalfunct$ is quasi-regular at $ \{0,\infty\} \cap \widetilde{\sigma}(A)$.
			\item $\finalfunct$ has exactly polynomial limits at $\{0,\infty\} \cap \widetilde{\sigma}(A) \cap \finalfunct^{-1}(\{0,\infty\})$.
		\end{enumerate}   
		Then, $\finalfunct(A)$ is sectorial of angle $\finalangle$.
	\end{theorem}

	\section{Some generation results of holomorphic semigroups and their properties}\label{Sec4}

	As a consequence of the bijection between generators of bounded holomorphic semigroups and sectorial operators,  the results obtained in Section \ref{sectorialSection} encourage us to study the properties related to the holomorphic semigroup generated by $-\finalfunct(A)$ whenever $A$ is a bisectorial operator,  and $\finalfunct$ a meromorphic function satisfying the hypothesis of Theorem \ref{sectorialTheorem} with an angle strictly smaller than $\frac{\pi}{2}$. We state explicitly this connection in the following result.
	
	\begin{corollary}\label{holomorphicCor}
		Let $A, \finalangle,\finalfunct$ be as in Theorem \ref{sectorialTheorem}. In addition, assume that $\finalangle \in [0,\pi/2)$.	Then, $-\finalfunct(A)$ generates a bounded holomorphic semigroup $T_\finalfunct$ of angle $\frac{\pi}{2} - \finalangle$.
	\end{corollary}
	
	\begin{proof}
		This is an immediate consequence of Theorem \ref{sectorialTheorem}	and the fact that an operator $B$ is sectorial of angle $\finalangle < \frac{\pi}{2}$ if and only if $-B$ is the generator of a bounded holomorphic semigroup of angle $\frac{\pi}{2}-\finalangle$, see for example \cite[Theorem 4.6]{engel2000one} or \cite[Proposition 3.4.4]{haase2006functional}.
	\end{proof}

	
	The lemma below will be useful in the proof of our main results within the case where either $aI-A$ or $aI+A$ is not injective. Its proof is analogous to the related result for sectorial operators (see e.g.  \cite[Lemma 4.3]{haase2005general}).
	
	\begin{lemma}\label{notInjectivityLemma}
		Let $A \in \BSect(\omega,a),\, g \in \mathcal{M}_A$ with quasi-regular limits at $M_A$, and assume that $d \in \sigma_p(A)$. Then, $g$ has a finite limit $c_d$ at $d$.
	\end{lemma}
	
	\begin{proof}
		First, we have that the limit $c_d$ of $g$ at $d$ exists in $\CC_\infty$. Indeed, this is clear if $d \notin M_A$. And if $d \in M_A$, by hypothesis $g$ is quasi-regular at $d \in M_A$, and therefore has a limit $c_d \in \CC_\infty$.
		
		Second, let $e \in \mathcal{E}$ be a regularizer for $g$, so $eg \in \mathcal{E}$ and $e(A)$ is injective. One can easily check that $e(A)x = e(d)x$ for all $x \in \mathcal{N}(dI-A)$. This implies that $e(d) \neq 0$. Since $eg \in \mathcal{E}$,  we have that $g$ has a finite limit $c:=g(a)$ at $a$,  and the assertion follows.
	\end{proof}
	
	Recall that the space of strong continuity $\mathbb{D}_T$ of a (holomorphic) semigroup $T$ generated by $A$ is precisely $\overline{\mathcal{D}(A)}$. The following result characterizes the space $\mathbb D_T$ in our setting.  Let us point out that the result holds even if the angle of sectoriality $\finalangle$ of $\finalfunct(A)$ is greater or equal than $\frac{\pi}{2}$.
	
	\begin{proposition}\label{domainProp}
		Let $A, \finalfunct$ be as in Theorem \ref{sectorialTheorem}. If $ \finalfunct^{-1}(\infty) \cap M_A = \emptyset$, then $\overline{\mathcal{D}(\finalfunct(A))} = X$. Otherwise,
		\begin{align*}
			\overline{\mathcal{D}(\finalfunct(A))} = \bigcap_{d \in \finalfunct^{-1}(\infty) \cap M_A} \overline{\mathcal{R}(dI-A)},
		\end{align*}
		where $\mathcal{R}(\infty I-A) := \mathcal{D}(A)$.
	\end{proposition}
	
	\begin{proof}
		First of all, notice that if $\finalfunct^{-1}(\infty) \cap M_A = \emptyset$, then $\finalfunct^{-1}(\infty) = \emptyset$ (see Remark \ref{openMappingRemark}), so by the inclusion of the spectra (Proposition \ref{spectralInclusionProposition}), $\finalfunct(A) \in \mathcal{L}(X)$ and $\mathcal{D}(\finalfunct(A)) = X$.
		
		Let $d \in M_A$. Note that if $d \notin \finalfunct^{-1}(\infty)$, then $\finalfunct$ is regular at $d$ with $\finalfunct(d) \in \CC$. For any $b>a$, consider
		\begin{equation*}
			f_d(z) := \begin{cases}
				\displaystyle	\finalfunct(a)\frac{b^2-a^2}{2a}\frac{a+z}{b^2-z^2}, \quad & \mbox{ if }d = a,
				\\ \displaystyle \finalfunct(-a)\frac{b^2-a^2}{2a}\frac{a-z}{b^2-z^2}, \quad &  \mbox{ if }d = -a,
				\\ \displaystyle \finalfunct(\infty)\frac{a^2-z^2}{b^2-z^2}, \quad&  \mbox{ if }d = \infty.
			\end{cases}
		\end{equation*}
		Then, $\finalfunct-f_d$ is regular at $d$ with limit $0$, and the behavior of $\finalfunct - f_d$ at $M_A\backslash \{d\}$ remains the same as the behavior of $\finalfunct$ at those points. Moreover, since $f_d(A) \in \mathcal{L}(X)$, it follows that 
		\begin{align*}
			\mathcal{D}(\finalfunct(A)) =  \mathcal{D}( \finalfunct_\bullet(A)) \;\mbox{ where }\;  \finalfunct_\bullet(z):=\finalfunct(z)-\sum_{d \notin \finalfunct^{-1}(\infty) \cap M_A} f_d(z).
		\end{align*}
		Thus, we can assume that $\finalfunct$ has regular limits equal to $0$ at $M_A \backslash \finalfunct^{-1}(\infty)$. 
		
		Now, we proceed by showing both inclusions $\subset, \supset$ of the statement, starting with the latter one. For all $t>0$ small enough (for which $b \notin \sigma(tA)$), set
		\begin{align}\label{regularisingFunction}
			h_t(z) := \frac{(a-z)^{n_a} (a+z)^{n_{-a}} b^{n_\infty}}{(t+a-z)^{n_a} (t+a+z)^{n_{-a}} (b+tz)^{n_\infty}}, \quad z \in \mathcal{D}(\finalfunct),
		\end{align}
		with $n_d = 0$ if $d \notin \finalfunct^{-1}(\infty)\cap M_A$ and the rest $n_d\in \NN$ large enough so that $h_t \finalfunct \in \mathcal{E}$. Then $h_t \finalfunct(A) \in \mathcal{L}(X)$ with $\mathcal{D}((h_t \finalfunct)(A)) = X$, and note that $h_t^{-1} \in \mathcal{M}_A$ since $\sigma_p(A) \cap \finalfunct^{-1}(\infty) = \emptyset$ (see Lemma \ref{notInjectivityLemma}). Therefore, $\finalfunct(A) \supset (h_t \finalfunct)(A) h_t^{-1}(A)$, which implies that $\mathcal{D}(\finalfunct(A)) \supset \mathcal{D}(h_t^{-1}(A)) = \mathcal{R}(h_t(A))$ for all $t>0$ small enough. 
		In addition, since both $aI+A$ and $aI-A$ are sectorial operators, we have that
		\begin{align*}
			\lim_{t\to 0} h_t(A) x = x, \quad \text{for all }  x \in \bigcap_{d \in g^{-1}(\infty) \cap M_A} \overline{\mathcal{R}(dI-A)},
		\end{align*}
		see \cite[Proposition 2.1.1 (c)]{haase2006functional}, which yields that 
		$$\bigcap_{d \in \finalfunct^{-1}(\infty) \cap M_A} \overline{\mathcal{R}(dI-A)} \subset \overline{\mathcal{R}(h_t(A))} \subset \overline{\mathcal{D}(\finalfunct(A))},
		$$
		and the inclusion $\supset$ of the assertion follows.
		
		Let us prove the reverse inclusion $\subset$. Assume that $\infty \in \finalfunct^{-1}(\infty) \cap M_A$. Then,  $|\finalfunct(z)| \sim |z|^\alpha$ as $z \to \infty$ for some $\alpha>0$. It follows that $(1+\finalfunct(z))^{-1}$ regularizes $(z+a)^{\alpha'}$ for any $\alpha' \in (0,\alpha)$, which implies that $\mathcal{D}(\finalfunct(A)) \subset \mathcal{D}((A+aI)^{\alpha'})$. Reasoning similarly with $-a,a$, one obtains that for any $\alpha'>0$ small enough,
		\begin{align*}
			\mathcal{D}(\finalfunct(A)) \subset \bigcap_{d \in \finalfunct^{-1}(\infty) \cap M_A} \mathcal{R}((dI-A)^{\alpha'}),
		\end{align*}
		where $\mathcal{R}((\infty I-A)^{\alpha'}) := \mathcal{D}((aI+A)^{\alpha'})$. Then, our proof will finish if we show that $\overline{\mathcal{R}((dI-A)^{\alpha'})} \subset \overline{\mathcal{R}(dI-A)}$. Assume that $d=\infty$. It follows from Theorem \ref{sectorialTheorem} that $(aI+A)^{\alpha'}$ is a  sectorial operator for a small enough $\alpha'>0$. Moreover, $aI+A$ is also a sectorial operator, and $(aI+A)^{\alpha'} = f_{\alpha'}(aI+A)$ (where $f(z)=z^{\alpha'}$) by using the NFC of sectorial operators (see \cite[Section 2.3]{haase2006functional}). Then, by the composition rule for sectorial operators (see e.g. \cite[Theorem 2.4.2]{haase2006functional}), one has that $f_{1/\alpha'}((aI+A)^{\alpha'}) = aI+A$.  Reasoning analogously as in the proof of the inclusion $\supset$, one gets that
		\begin{align*}
			\overline{\mathcal{D}(aI+A)} = \overline{\mathcal{D}((f_{1/\alpha'}) ((aI+A)^{\alpha'}))} \supset \overline{\mathcal{D}((aI+A)^{\alpha'})},
		\end{align*}
		as we wanted to prove. The cases $d \in \{-a,a\}$ are solved in an analogous way, by using the operators $(aI+A)^{-\alpha'}, (aI-A)^{-\alpha'}$,  respectively. The proof is finished.
	\end{proof}
	
	\begin{corollary}\label{domainReflexiveCor}
		If $X$ is reflexive, then $\overline{\mathcal{D}(g(A))} = X$.
	\end{corollary}
	
	\begin{proof}
		By \cite[Proposition 2.1.1 (h)]{haase2006functional}, one has that 
		$$X = \overline{\mathcal{D}(A)} = \mathcal{N}(aI-A) \oplus \overline{\mathcal{R}(aI-A)} = \mathcal{N}(aI+A) \oplus \overline{\mathcal{R}(aI+A)}$$
		if $X$ is reflexive. Since $\sigma_p(A) \cap g^{-1}(\infty) = \emptyset$ (see Lemma \ref{notInjectivityLemma}), the statement follows by an application of Proposition \ref{domainProp}.
	\end{proof}
	
	Recall that $\exp_{-w} (z):= \exp(-wz)$ for all $z,\,w\in \CC$. Since $T_\finalfunct(w) = \exp_{-w} (\finalfunct(A))$, it is natural to conjecture that $T_\finalfunct(w) = (\exp_{-w} \circ \finalfunct) (A)$. The theorem below answers this question positively. Its proof is inspired by the composition rule for sectorial operators given in \cite{haase2005general}, but carefully adapted to cover all our cases. Indeed, one could easily generalize the result below to a composition rule from bisectorial to sectorial operators, addressing a larger class of functions. However, this would require to introduce several new definitions and additional cumbersome notations. Thus, in order for the paper to be accessible for a broad class of mathematicians, we will limit to our specific cases. 
	
	\begin{theorem}\label{compositionRule}
		Let $\finalangle, A, g$ be as in Corollary \ref{holomorphicCor}, so that $-g(A)$ generates a holomorphic semigroup $T_g$ of angle $\frac{\pi}{2} - \finalangle$. Then, for any $w \in S_{\pi/2-\finalangle}$,  we have that $(\exp_{-w} \circ g) \in \mathcal{M}_A$ and 
		\begin{align}\label{exponentialFormula}
			T_g(w) = (\exp_{-w} \circ g)(A).
		\end{align}
	\end{theorem}

	\begin{proof}
		First of all, the claim is trivial if $\finalfunct=0$, so we will assume that $\finalfunct\neq 0$. Fix  $w \in S_{\pi/2-\finalangle}$. Then, it is straightforward to check that $(\exp_{-w} \circ \finalfunct)$ is regular at $M_A$, so Lemma \ref{differentNFCLemma} yields that $(\exp_{-w} \circ \finalfunct) \in \mathcal{M}(A)$. 
		
		Now set $f_w (z) := \exp_{-w} (z) - (1+z)^{-1}$. Then,  $f_w \in \mathcal{E}_0[S_{\finalangle}]$. As $-1 \notin \sigma(\finalfunct(A))\subset \overline{S_\finalangle}$, an application of Lemma \ref{functionalCalculusProperties} (f) yields that $(f_w \circ \finalfunct)  \in \mathcal{M}_A$ and $(I+\finalfunct)^{-1}(A) = (I+\finalfunct(A))^{-1}$. Therefore, our statement will follow if we prove that $(f_w \circ \finalfunct)(A) = f_w(\finalfunct(A))$. 
		
		Recall that, for $d \in \CC$, we denote by $c_d$ the limit of $\finalfunct(z)$ as $z \to d$ whenever it exists. In particular, it exists if $d \in \sigma_p(A)$, see Lemma \ref{notInjectivityLemma}. Let $b>a$, and for any $\lambda \notin \overline{S_\finalangle}$, set
		\begin{align*}
			G_\lambda (z):= \frac{1}{\lambda - \finalfunct(z)} - \sum_{d \in \sigma_p(A) \cap \{-a,a\}} \frac{1}{\lambda-c_d} \frac{z+d}{b-z}\frac{b-d}{2d}, \quad z \in \mathcal{D}(\finalfunct).
		\end{align*}
		Since $\lambda \in \rho(A)$, one has that $G_\lambda \in \mathcal{M}_A$. Moreover, $G_\lambda \in H^\infty (\mathcal{D}(\finalfunct))$ with $G_\lambda (d) = 0$ for all $d \in \sigma_p(A) \cap \{-a,a\}$. Furthermore,  it is readily seen that one can find a regularizer $e \in \mathcal{E}$ independent of $\lambda$, for which $eG_\lambda \in \mathcal{E}_0$. Indeed, to check the regularity of $eG_\lambda$ at the points $d' \in M_A$, one can add to $e$ powers of the function $(z-d')/(z-b)^2$ if $d'\notin \sigma_p(A)$. Otherwise, the regularity is obtained by the bounds in Lemma \ref{IneqLemma} (recall that in this case, $c_{d'} \neq \infty$ by Lemma \ref{notInjectivityLemma}).
		
		Then, let $\Gamma'$ be an appropriate path for the NFC of the sectorial operator $\finalfunct(A)$ and the function $f_w$. 
		It follows that
		\begin{align*}
			f_w(\finalfunct(A)) &= e(A)^{-1} e(A) f_w (\finalfunct(A))
			= e(A)^{-1}  \frac{1}{2\pi i}  \int_{\Gamma'} f_w(\lambda) e(A) \finalfunct (A) \, d\lambda
			\\ &= e(A)^{-1}  \frac{1}{2\pi i}  \int_{\Gamma'} f_w(\lambda) (e(z)G_\lambda (z))(A) \, d\lambda
			\\ & \quad +  \sum_{d \in \sigma_p(A) \cap \{-a,a\}} \frac{b-d}{2d}  (dI+A)R(b,A)\; \frac{1}{2\pi i}  \int_{\Gamma'}  \frac{f_w(\lambda)}{\lambda-c_d}  \, d\lambda.
		\end{align*}
		By Cauchy's integral theorem, one has that the last term is precisely 
		\begin{align*}
			\sum_{d \in \sigma_p(A) \cap \{-a,a\}} \frac{b-d}{2d} f(c_d) (dI+A)R(b,A) .
		\end{align*}
		Now, let us compute the first term. Let $\Gamma$ be an appropriate path of the NFC of the bisectorial operator $A$. Since $eG_\lambda \in \mathcal{E}_0$, one has that
		\begin{align}\label{iden}
			e(A)^{-1} \frac{1}{2\pi i}  \int_{\Gamma'} f_w(\lambda) (e(z)G_\lambda (z))(A) \, d\lambda
			&= e(A)^{-1} \frac{1}{(2\pi i)^2}  \int_{\Gamma'} f_w(\lambda) \int_\Gamma e(z) G_\lambda (z) R(z,A)\, dz d\lambda\notag
			\\ &= e(A)^{-1} \frac{1}{(2\pi i)^2}  \int_\Gamma e(z) R(z,A)  \int_{\Gamma'} f_w(\lambda) G_\lambda (z) \, d\lambda dz.
		\end{align}
		Let us go on with the proof before checking the hypothesis for Fubini's theorem that we have applied in the last equality in \eqref{iden}.  By Cauchy's theorem, it easily follows that
		\begin{align*}
			\frac{1}{2\pi i}\int_{\Gamma'} f_w(\lambda) G_\lambda (z) \, d\lambda = f_w(\finalfunct(z)) - \sum_{d \in \sigma_p(A) \cap \{-a,a\}} f(c_d) \frac{z+d}{b-z}\frac{b-d}{2d}.
		\end{align*}
		From this,  we can conclude that in fact 
		\begin{align*}
			&e(A)^{-1} \frac{1}{(2\pi i)^2}  \int_{\Gamma} e(z) R(z,A)  \int_\Gamma f_w(\lambda) G_\lambda (z) \, d\lambda dz
			\\ & \quad = (f_w \circ \finalfunct)(A) - \sum_{d \in \sigma_p(A) \cap \{-a,a\}} \frac{b-d}{2d} f(c_d) (dI+A)R(b,A),
		\end{align*}
		and our assertion follows.
		
		Let us check now that indeed Fubini's theorem can be applied. For that, we have to check the integrability of the function
		\begin{align*}
			F(\lambda, z) := \frac{f_w (\lambda)}{\lambda} \lambda G_\lambda (z) \frac{e(z)}{\min\{|z-a|,|z+a|\}},
		\end{align*}
		on $\Gamma' \times \Gamma$. First, $f_w(\lambda)/\lambda$ is clearly integrable on $\Gamma'$ and, by Lemma \ref{IneqLemma}, $\lambda G_\lambda (z)$ is uniformly bounded on $\Gamma' \times \Gamma$. Now, one can assume that $\frac{e(z)}{\min\{|z-a|,|z+a|\}}$ is integrable on $\Gamma$ if and only if $\{-a,a\} \cap \sigma_p(A) = \emptyset$. Otherwise, one has to check a uniform bound for the integral of $F(\lambda,z)$ on the intersection of $\Gamma$ with a neighborhood of $d \in \sigma_p(A)\cap\{-a,a\}$.
		
		So let $d \in \sigma_p(A) \cap \{-a,a\}$. Recall that in this case, $c_d \in \overline{S_\finalangle}$ with $c_d \neq \infty$. If $c_d \neq 0$, then $\lambda G_\lambda$ is of the same type as the function appearing in Step 1 in the proof of Proposition \ref{existenceProposition}, and proceeding as there, one can easily check the integrability condition.
		
		Thus,  we can assume that $c_d = 0$. So,  one has that
		\begin{align*}
			|\lambda G_\lambda (z)| \lesssim \frac{|\finalfunct(z)|}{|\lambda - \finalfunct(z)|} + |z-d|, \qquad \mbox{ as }\; z \to d,
		\end{align*}
		where the $|z-d|$ term is the result of applying a Taylor expansion of order $1$ in a similar way as in Step 2 in the proof of Proposition \ref{existenceProposition}. It is readily seen that the $|z-d|$ term does not entangle the bound of $F(\lambda,z)$. Moreover, for any $\delta \in (0,1)$, one has that
		\begin{align*}
			&\left|\frac{f_w (\lambda)}{\lambda} \frac{\finalfunct(z)}{\lambda - \finalfunct(z)} \frac{e(z)}{z-d}\right| = \left|\frac{f_w (\lambda)}{\lambda^{1+\delta}}\right| \left|\frac{\lambda^\delta \finalfunct(z)^{1-\delta}}{\lambda - \finalfunct(z)}\right| \left|\frac{(e\finalfunct^\delta)(z)}{z-d}\right|.
		\end{align*}
		It is easy to see that $f_w(\lambda)/\lambda^{1+\delta}$ is still integrable on $\Gamma'$, and that the middle term is uniformly bounded. Moreover, since $c_d = 0$, we have by hypothesis that $|\finalfunct(z)| \sim |z-d|^\alpha$ as $z\to d$ for some $\alpha > 0$. Thus $\finalfunct^\delta (z) \lesssim |z-d|^{\alpha \delta}$, so the last term is integrable in $\Gamma$.  The proof is finished.
	\end{proof}

	\section{Generalized Black-Scholes equations on interpolation spaces}\label{BlackScholesSection}
	
	In this section, we apply the theory developed in the previous sections to introduce and study generalized Black--Scholes equations on $(L^1-L^\infty)$-interpolation spaces.
	Throughout the following, without any mention,  $E$ will denote a $(L^1,L^\infty)$-interpolation space on $(0,\infty)$. 
	We recall that 
	$$(G_E(t) f)(x) := f(e^{-t}x), \qquad x> 0, \, t \in \RR,\, f \in E,$$
	defines a group of bounded operators $G_E=(G_E(t))_{t\in\RR}$ on $E$ with $\|G_E(t)\|_{\mathcal L(E)} \leq \max\{1,e^t\}$ for $t \in \RR$, and which is strongly continuous if and only $E$ has order continuous norm. Then, the lower and upper Boyd indices $\underline{\eta}_E, \overline{\eta}_E$ are defined by
	\begin{align*}
		\underline{\eta}_E := - \lim_{t\to\infty} \frac{\log \|G_E(-t)\|_{\mathcal L(E)}}{t},
		\quad \overline{\eta}_E := \lim_{t\to\infty} \frac{\log \|G_E(t)\|_{\mathcal L(E)}}{t},
	\end{align*}
	and they satisfy $0 \leq \underline{\eta}_E \leq \overline{\eta}_E \leq 1$. 
	
	The generator $J_E$ of the group $G_E$ is given by
	\begin{equation}\label{gene}
	\begin{cases}
	\mathcal{D}(J_E) =\Big \{f \in E \, : f \in AC_{\rm loc}(0,\infty) \text{ and } -xf'(x) \in E\Big\},\\
		(J_E f)(x) := -xf'(x), \quad x>0, \, f \in \mathcal{D}(J_E), 
		\end{cases}
	\end{equation}
	and its spectrum is given by 
	$$\sigma(J_E) = \{\lambda \in \CC \, : \, \underline{\eta}_E \leq \Re \lambda \leq \overline{\eta}_E\},$$
	see \cite{arendt2002spectrum} for more details about the operator $J_E$ and the classical Black--Scholes equation on exact $(L^1-L^\infty)$-interpolation spaces. 
	However, as the authors indicate in \cite{arendt2002spectrum},  every $(L^1-L^\infty)$-interpolation space can be equivalently renormed so that it becomes exact (e.g. \cite[Proposition III.1.13]{bennett1988interpolation}). Thus, we may apply the results in \cite{arendt2002spectrum} to arbitrary $(L^1-L^\infty)$-interpolation spaces, without requiring them to be exact.
	
	As a consequence of the above properties, for any $\underline{\varepsilon}, \overline{\varepsilon} >0$, both $\left(\underline{\eta}_E + \underline{\varepsilon}\right)I + J_E$ and $\left(\overline{\eta}_E +\overline{\varepsilon}\right)I - J_E$ are sectorial operators of angle $\frac{\pi}{2}$ (see e.g.  \cite[Section 2.1.1]{haase2006functional}). Therefore, $J_E - \frac{\overline{\eta}_E + \underline{\eta}_E + \overline{\varepsilon} - \underline{\varepsilon}}{2}I$ is a bisectorial operator of angle $\pi/2$ and half-width $\frac{\overline{\eta}_E - \underline{\eta}_E + \overline{\varepsilon} + \underline{\varepsilon}}{2}$. However, to avoid cumbersome notations we will write $f(J_E)$ to refer to $f_k(J_E-k)$ for $k = \frac{\overline{\eta}_E + \underline{\eta}_E + \overline{\varepsilon} - \underline{\varepsilon}}{2}$ and $f_k(z)=f(z+k)$. Notice that one may take $\underline{\varepsilon}=\overline{\varepsilon} = 0$ if $\underline{\eta}_E =0$ and $\overline{\eta}_E = 1$,  respectively, or if $E = L^p$ with $1\leq p \leq \infty$.
	

	
	In \cite{arendt2002spectrum},  the authors make use of the operator $J_E$ to study the classical Black-Scholes partial differential equation in $(L^1,L^\infty)$-interpolation spaces. Recall that the classical Black-Scholes equation is the degenerate parabolic equation given by
	\begin{align}\label{BSEquation}
		u_t = x^2 u_{xx} + x u_x, \quad x,t > 0.
	\end{align}
	In fact,  we can rewrite \eqref{BSEquation} as $u_t = J_E^2 u$, where $J_E$ is the operator defined in \eqref{gene}. 
	
	Next, we introduce the fractional operators that generalize the Black--Scholes equation \eqref{BSEquation}.   On the one hand, we will consider fractional powers of the operator $J_E$. If $\alpha \in (0,n)$,  $n\in\mathbb N$, it follows that $\mathcal{D}(J_E^n) \subset \mathcal{D}(J_E^\alpha)$ (see \cite[Proposition 3.1.1 ]{haase2006functional}). If in addition $0<\alpha < 1$, an application of Fubini's theorem to the Balakrishnan representation of $J_E^{\alpha} f$ together with the resolvent identity yields that, whenever $\underline{\eta}_E >0$,
	\begin{align*}
		(J_E^{\alpha} f)(x) = \frac{-1}{\Gamma(1-\alpha)} \int_x^\infty \left(\log \frac{s}{x}\right)^{-\alpha} f'(s) \, ds, \quad f \in \mathcal{D}(J_E), \, x>0.
	\end{align*} 
	If $\underline{\eta}_E = 0$, then one cannot apply Fubini's theorem to obtain the above expression.  However, one can use the fact that $(J_E+\varepsilon I)^{\alpha} f \to J_E^{\alpha} f$ in $E$ as $\varepsilon \downarrow 0$ (see \cite[Proposition 3.1.9]{haase2006functional}), together with
	\begin{align}\label{balakrishnanRepr}
		((J_E+\varepsilon I)^{\alpha} f)(x) = \frac{-1}{\Gamma(1-\alpha)} \int_x^\infty \left(\log \frac{s}{x}\right)^{-\alpha} \left(\frac{x}{s}\right)^\varepsilon f'(s) \, ds,
	\end{align} 
	for any $f \in \mathcal{D}(J_E)$ and $x,\varepsilon > 0$.

	
	Next, let $\alpha > 0$ be a real number and recall that we denote by $D^{-\alpha}$  the Riemann-Liouville fractional integral of order $\alpha$, and by $W^{-\alpha}$ the Weyl fractional integral of order $\alpha$, see \eqref{RLFI} and \eqref{WFI},  respectively.
	Similarly, $D^{\alpha}$ denotes the Riemann-Liouville fractional derivative of order $\alpha$, and $W^\alpha$ the Weyl fractional derivative of order $\alpha$, defined in \eqref{RLFD} and \eqref{WFD}, respectively.
	Also, if $m^s$ is the multiplication operator by $x^s$ for any $s \in \RR$, we have that the generalized Cesàro operator $\mathcal{C}_\alpha$ of order $\alpha$, and its adjoint $\mathcal{C}_\alpha^\ast$, are given, respectively,  by
	\begin{align*}
		\mathcal{C}_\alpha = \Gamma(\alpha +1) m^{-\alpha} I^{-\alpha}\;\mbox{ and } \quad  &\mathcal{D}^\alpha = (\mathcal{D}^{-\alpha})^{-1},
		\mathcal{C}_\alpha^\ast = \Gamma(\alpha +1)  W^{-\alpha} m^{-\alpha}, 
	\end{align*}
	see \eqref{FC1} and \eqref{FC1}. Note that these operators are injective due to the fact the operators $\riemannDerivative^{-\alpha}, W^{-\alpha}$ and $m^{-\alpha}$ are injective. Moreover, by equality \eqref{FC1} again, one has that $\mathcal{C}_\alpha$ defines a bounded operator $\mathcal{C}_{\alpha,E}$ when restricted to our $(L^1-L^\infty)$-interpolation space $E$ with $\overline{\eta}_E<1$, since
	\begin{align*}
		\mathcal{C}_{\alpha,E} = \alpha \int_0^\infty e^{-s}(1-e^{-s})^{\alpha-1}  G_E (s)\, ds = \alpha \mathbb B(I-J_E, \alpha),
	\end{align*}
	where we have applied Proposition \ref{HilleStrip}, and $\mathbb B$ denotes the usual Beta function. 
	
	Similarly, $\mathcal{C}_\alpha^\ast$ defines a bounded operator $\mathcal{C}_{\alpha,E}^\ast$ on any $(L^1-L^\infty)$-interpolation space $E$ with $\underline{\eta}_E>0$, satisfying $\mathcal{C}_{\alpha,E}^\ast = \alpha \mathbb B(J_E, \alpha)$. As a consequence, one obtains that 
	$$\mathcal{\riemannDerivative}_E^\alpha := (\mathcal{C}_{\alpha,E})^{-1} = \left(\alpha \mathbb B(I-J_E, \alpha)\right)^{-1}\mbox{ and } \mathcal{W}_E^\alpha := (\mathcal{C}_{\alpha,E}^\ast)^{-1} = \left(\alpha \mathbb B(J_E, \alpha)\right)^{-1},$$
	are closed operators on $E$ whenever $\overline{\eta}_E < 1$ and $\underline{\eta}_E>0$,  respectively. The above identities  appear in \cite{lizama2014boundedness} for a family of Sobolev spaces on $(0,\infty)$.


	\subsection{Generation results  of holomorphic semigroups of fractional powers operators}
	The identity $J_E = \mathcal{W}_E^1 = I - \mathcal{\riemannDerivative}_E^1$ holds whenever the operators are well defined on $E$ (see e.g.  \cite{arendt2002spectrum}). In particular, we have that 
	\begin{align}\label{opera-1}
		(J_E)^{2} = (I - \mathcal{\riemannDerivative}_E^1)^2 = (\mathcal{W}_E^1)^2 =  \mathcal{W}_E^1 (I - \mathcal{\riemannDerivative}_E^1).
	\end{align}
	This motivates us to study different fractional versions of the Black--Scholes equation \eqref{BSEquation}.  In this section, we make use of the theory we developed in the preceding sections to obtain that the operators 
	\begin{align}\label{opera}
		(J_E)^{2\alpha},\;\;(I - \mathcal{\riemannDerivative}_E^\alpha)^2,\;\;  (\mathcal{W}_E^\alpha)^2,\;\;  \mathcal{W}_E^\alpha (I - \mathcal{\riemannDerivative}_E^\alpha),
	\end{align}
	are indeed generators of exponentially bounded holomorphic semigroups  on $E$ for suitable values of $\alpha$.
	
	We start with the operator $(J_E)^{2\alpha} $. 
	
	\begin{proposition}\label{fractionalPowerProposition}
		Let $E$ be a $(L^1-L^\infty)$-interpolation space, $n \in \NN$ and $\alpha\in \left(n-\frac{1}{2}, n+\frac{1}{2}\right)$. Then, the operator $(-1)^{n+1} (J_E) ^{2\alpha}$ generates an exponentially bounded holomorphic semigroup $T_{(-1)^{n+1} (J_E)^{2\alpha}}$ of angle $\pi\left(\frac{1}{2}  - \left|\alpha - n\right|\right)$, which is given by
		\begin{align*}\label{fractionalPowerBSFormula}
			\left(T_{(-1)^{n+1} (J_E)^{2\alpha}} (w)f\right)(x) = \frac{1}{2\pi} \int_0^\infty \frac{f(s)}{s} \int_{-\infty}^\infty \left(\frac{s}{x}\right)^{iu} \exp( (-1)^{n+1} w u^{2\alpha}) \, du ds, \quad x > 0, 
		\end{align*}
		for any $w \in S_{\pi\left(\frac{1}{2}  - |\alpha - n|\right)}$ and $f \in E$. In addition, $\overline{\mathcal{D} ((J_E)^{2\alpha})} = \overline{\mathcal{D} (J_E)}$.
	\end{proposition}
	
	\begin{proof}
		That the operator $(-1)^{n+1} (J_E)^{2\alpha}$ generates an exponentially bounded  holomorphic semigroup with the given angle follows from Corollary \ref{StripCor}. The expression given for $T_{(-1)^{n+1} (J_E)^{2\alpha}}$ is an immediate consequence of Theorem \ref{compositionRule} and Proposition \ref{HilleStrip}. The assertion about $\overline{\mathcal{D} ((J_E)^{2\alpha})}$ follows from Proposition \ref{domainProp}.
	\end{proof}
	
	Next, we have the following result for the operator $(I-\mathcal{\riemannDerivative}_E^\alpha)^2$.
	
	\begin{proposition}\label{CaProposition}
		Let $E$ be a $(L^1-L^\infty)$-interpolation space with $\overline{\eta}_E < 1$, $n \in \NN$ and $\alpha\in \left(n-\frac{1}{2}, n+\frac{1}{2}\right)$. Then,  the operator $(-1)^{n+1} (I-\mathcal{\riemannDerivative}_E^\alpha)^2$ generates an exponentially bounded holomorphic semigroup $T_{(-1)^{n+1} (I-\mathcal{\riemannDerivative}_E^\alpha)^2}$ of angle $\pi\left(\frac{1}{2}  - |\alpha - n|\right)$, which is given by
		\begin{align*}
			\left(T_{(-1)^{n+1}(I-\mathcal{\riemannDerivative}_E^\alpha)^2}(w)f\right)(x) &= \frac{1}{2\pi} \int_0^\infty \frac{f(s)}{s} 
			\int_{-\infty}^\infty \left(\frac{s}{x}\right)^{iu} \exp\left((-1)^{n+1} w \left(1-\frac{1}{\alpha \mathbb B(1-iu,\alpha)}\right)^2\right) \, du ds, 
		\end{align*}
		for $ x > 0, \, w \in S_{\pi\left(\frac{1}{2}  - |\alpha - n|\right)}$ and $f \in E$. In addition, $\overline{\mathcal{D} ((I-\mathcal{\riemannDerivative}_E^\alpha)^2)} = \overline{\mathcal{D} (J_E)}$.
	\end{proposition}
	
	\begin{proof}
		First, recall that $\mathcal{\riemannDerivative}_E^\alpha = (\alpha \mathbb B(I-J_E,\alpha))^{-1}$, so $(I- \mathcal{\riemannDerivative}_E^\alpha)^2 = (I - \alpha \mathbb B(I-J_E,\alpha))^{-1}$. It follows that
		$$\left(1-\frac{1}{\alpha \mathbb B(1-z,\alpha)}\right)^2 = \left(1 - \frac{1}{\Gamma(\alpha+1)} \frac{\Gamma(1 + \alpha-z)}{\Gamma(1-z)}\right)^2,$$ 
		which is holomorphic in $\CC \backslash \{1,2,3,...\}$. In addition, for $\lambda,z \in \CC$, one has that
		\begin{equation}\label{gammaAsymptotic}
			\frac{\Gamma(z+\lambda)}{\Gamma(z)} = z^\lambda \left(1+O(|z|^{-1})\right), \quad \mbox{ as } |z| \to \infty,
		\end{equation}
		whenever $z\neq 0, -1, -2,...$ and $z \neq -\lambda, -\lambda-1, -\lambda-2...$, (see e.g.  \cite{tricomi1951asymptotic} for more details). As a consequence, one gets that 
		\begin{align*}
			\left(1-\frac{1}{\alpha \mathbb B(1-z,\alpha)}\right)^2 = \frac{(-z)^{2\alpha}}{\alpha} \left(1+O(|z|^{-1})\right), \quad \mbox{ as } |z| \to \infty.
		\end{align*}
		Thus, for any $\finalangle \in \left(0,\pi\left(\frac{1}{2}  - |\alpha - n|\right)\right)$, one can find a $\rho>0$ large enough such that the function $\rho + (-1)^{n+1} \left(1-\frac{1}{\alpha \mathbb B(1-z,\alpha)}\right)^2$ satisfies the hypothesis of Corollary \ref{holomorphicCor}, i.e. $(-1)^{n+1} (I-\mathcal{\riemannDerivative}_E^\alpha)^2$ generates an exponentially bounded holomorphic semigroup of angle $\pi\left(\frac{1}{2}  - |\alpha - n|\right)$. The rest of the statement follows by a similar reasoning as in the proof of Proposition \ref{fractionalPowerProposition}.
	\end{proof}
	
	We have the following generation result for the operator $(\mathcal{W}_E^\alpha)^2$.
	
	\begin{proposition}\label{CaastProposition}
		Let $E$ be a $(L^1-L^\infty)$-interpolation space with $\underline{\eta}_E > 0$, $n \in \NN$ and $\alpha \in  \left(n-\frac{1}{2}, n+\frac{1}{2}\right)$. Then, the operator $(-1)^{n+1}(\mathcal{W}_E^\alpha)^2$ generates an exponentially bounded holomorphic semigroup $T_{(-1)^{n+1} (\mathcal{W}_E^\alpha)^2}$ of angle $\pi\left(\frac{1}{2}  - |\alpha - n|\right)$, which is given by
		\begin{align*}
			\left(T_{(-1)^{n+1} (\mathcal{W}_E^\alpha)^2}(w)f\right)(x) &= \frac{1}{2\pi} \int_0^\infty \frac{f(s)}{s} 
			\int_{-\infty}^\infty \left(\frac{s}{x}\right)^{iu+\delta} \exp\left((-1)^{n+1} w \left(\alpha \mathbb{B}(iu+\delta,\alpha)\right)^{-2}\right) \, du ds,  
		\end{align*}
		for $x>0$, $w \in S_{\pi\left(\frac{1}{2}  - |\alpha - n|\right)}$,  and $f \in E$,  where $\delta$ is any number $\delta>0$. In addition, $\overline{\mathcal{D} ((\mathcal{W}_E^\alpha)^2)} = \overline{\mathcal{D} (J_E)}$.
	\end{proposition}
	
	\begin{proof}
		The proof is analogous to the proof of Proposition \ref{CaProposition}, using that $\mathcal{W}_E^\alpha = (\alpha \mathbb B(J_E,\alpha))^{-1}$. The only difference comes out that one cannot apply Cauchy's Theorem and translate the inner integral path in $u$ to make $\delta = 0$ since the Euler-Beta function $\mathbb B(0,\alpha)$ has an essential singularity for any non natural number $\alpha$.
	\end{proof}
	
	Finally, we have the following generation result for the operator $\mathcal{W}_E^\alpha (I-\mathcal{\riemannDerivative}_E^\alpha)$.
	
	\begin{proposition}\label{CamixProposition}
		Let $E$ be a $(L^1-L^\infty)$-interpolation space with $\underline{\eta}_E > 0$ and $\overline{\eta}_E < 1$, and let $\alpha>0$. Then, $\mathcal{W}_E^\alpha (I-\mathcal{\riemannDerivative}_E^\alpha)$ generates an exponentially bounded holomorphic semigroup $T_{\mathcal{W}_E^\alpha (I-\mathcal{\riemannDerivative}_E^\alpha)}$ of angle $\frac{\pi}{2}$, which is given by
		\begin{align*}
			&\left(T_{\mathcal{W}_E^\alpha (I-\mathcal{\riemannDerivative}_E^\alpha)}(w)f\right)(x) \\
			= &\frac{1}{2\pi} \int_0^\infty \frac{f(s)}{s} \int_{-\infty}^\infty \left(\frac{s}{x}\right)^{iu+\delta} \exp\left(\frac{w}{ \alpha \mathbb B(\delta+iu,\alpha)}\left(1-\frac{1}{\alpha \mathbb B(1-\delta - iu,\alpha)}\right)\right) \, du ds, 
		\end{align*}
		for $ x > 0, \, w \in S_{\frac{\pi}{2}}$,  and $f \in E$,  where $\delta \in (0,1)$ is any number.  In addition, $\overline{\mathcal{D} (\mathcal{W}_E^\alpha (I-\mathcal{\riemannDerivative}_E^\alpha))} = \overline{\mathcal{D} (J_E)}$.
	\end{proposition}
	
	\begin{proof}
		The proof is analogous to the proof of Propositions \ref{CaProposition} and \ref{CaastProposition}.  Here,  the statement is valid for any $\alpha >0$ since, by \eqref{gammaAsymptotic}, we have that
		\begin{align*}
			\frac{1}{ \alpha \mathbb B(z,\alpha)}\left(1-\frac{1}{\alpha \mathbb B(1-z,\alpha)}\right) = \frac{z^\alpha (-z)^\alpha}{2\alpha} (1 + O(|z|^{-1})), \quad \mbox{ as } |z| \to \infty. 
		\end{align*}
		The proof is finished.
	\end{proof}
	
	\subsection{Generalized Black-Scholes partial differential equations}
	
	Let $B_E$ be a closed linear operator on a Banach space $E$, and consider the following abstract Cauchy problem:
	\begin{align}\label{CauchyP}
		\tag{$ACP_0$}
		\begin{cases*}
			\displaystyle u \in C^1((0,\infty);E), \quad  u(t) \in \mathcal{D}(B_E), \quad t>0,
			\\ \displaystyle u' (t) =  B_E u (t), \qquad\,\,\, t>0,
			\\\displaystyle  \lim_{t\downarrow 0}u(t) = f \in E.
		\end{cases*}
	\end{align}
	We say that the Cauchy problem ($ACP_0$) is well-posed,  if for for any $f \in E$, there exists a unique solution $u$. 
	
	We are ready to state the following result concerning the well-posedness of the fractional Black-Scholes equation. Before that, let us state explicitly how these equations look like. Let $n \in \NN$,  $\alpha > 0$, and recall that $\riemannDerivative^\alpha$ and $W^\alpha$ denote, respectively, the Riemann-Liouville and Weyl fractional derivatives of order $\alpha$ acting on the spatial domain.
	\begin{enumerate}
		\item[(1)] In the case $B_E = (-1)^{n+1} (J_E)^{2\alpha}$ we have the following situation:
		\begin{itemize}
	\item 	If $\underline{\eta}_E > 0$, one can use the Balakrishnan representation, to obtain 
		\begin{equation*}
			(-1)^{n+1} u_t (x) = \frac{-1}{\Gamma(1-\alpha)} \int_x^\infty \left(\log \frac{s}{x}\right)^{-2\alpha+n} U_n'(s) \, ds, \quad t, x > 0.
		\end{equation*}
		
		\item If $\underline{\eta}_E = 0$, one has to proceed as in \eqref{balakrishnanRepr} to obtain 
		\begin{equation*}
			(-1)^{n+1} u_t (x)  =\lim_{\varepsilon \downarrow 0}\frac{-1}{\Gamma(1-\alpha)} \int_x^\infty \left(\log \frac{s}{x}\right)^{-2\alpha+n} \left(\frac{x}{s}\right)^\varepsilon U_n'(s) \, ds, \quad t, x > 0.
		\end{equation*}
	\end{itemize}
In both cases,  $n \in \NN$ is the whole part of $2\alpha$ and $U_n := (J_E)^n U$. 

		\item[(2)] If $B_E = (-1)^{n+1} (I - \mathcal{\riemannDerivative}_E^\alpha)^2$, one obtains the equation
		\begin{equation*}
			(-1)^{n+1} u_t = \frac{1}{\Gamma(\alpha+1)^2} \riemannDerivative^\alpha (x^\alpha \riemannDerivative^\alpha (x^\alpha u)) - \frac{2}{\Gamma(\alpha+1)} \riemannDerivative^\alpha (x^\alpha u) + u,\quad t,x>0.
		\end{equation*}
		\item[(3)] If $B_E = (-1)^{n+1} (\mathcal{W}_E^\alpha)^2$, one gets the equation
		\begin{equation*}
			(-1)^{n+1} u_t = \frac{1}{\Gamma(\alpha+1)^2} x^\alpha W^\alpha (x^\alpha W^\alpha u),\quad t,x>0.
		\end{equation*}
		\item[(4)] The case $B_E = \mathcal{W}_E^\alpha (I - \mathcal{\riemannDerivative}_E^\alpha)$ leads to the equation
		\begin{equation*}
			u_t = \frac{1}{\Gamma(\alpha+1)}x^\alpha W^\alpha u - \frac{1}{\Gamma(\alpha+1)^2} \riemannDerivative^\alpha (x^{2\alpha} W^\alpha u),\quad t,x>0.
		\end{equation*}
	\end{enumerate}
	
	We have the following result.
	
	\begin{theorem}\label{ACP0Theorem}
		Let $E$ be a $(L^1-L^\infty)$-interpolation space with order continuous norm, $n \in \NN$,  and $\alpha>0$. Then,  the following assertions hold.
		\begin{enumerate}[(a)]
			\item If $\alpha \in \left(n-\frac{1}{2}, n+\frac{1}{2}\right)$, then ($ACP_0$) is well-posed with $B_E = (-1)^{n+1} (J_E)^{2\alpha}$.
			\item If $\overline{\eta}_E < 1$ and $\alpha \in \left(n-\frac{1}{2}, n+\frac{1}{2}\right)$, then ($ACP_0$) is well-posed with $B_E = (-1)^{n+1} (I-\mathcal{\riemannDerivative}_E^\alpha)^2$.
			\item If $\underline{\eta}_E > 0$ and $\alpha\in \left(n-\frac{1}{2}, n+\frac{1}{2}\right)$, then ($ACP_0$) is well-posed with $B_E = (-1)^{n+1} (\mathcal{W}_E^\alpha)^2$.
			\item If $\overline{\eta}_E<1$ and $\underline{\eta}_E > 0$, then ($ACP_0$) is well-posed with $B_E = \mathcal{W}_E^\alpha (I-\mathcal{\riemannDerivative}_E^\alpha)$.
		\end{enumerate}
		In any case, the solution $u$ of ($ACP_0$) is given by $u(t) = T_{B_E}(t) f$ for $t > 0$. In addition, identifying $u(t,x) = u(t)(x)$, we obtain that $u \in C^\infty((0,\infty) \times (0,\infty))$.
	\end{theorem}
	
	\begin{proof}
		In all cases, $B_E$ is the generator of a holomorphic semigroup with $\overline{\mathcal{D}(B_E)} = \overline{\mathcal{D}(J_E)}$ by Propositions \ref{fractionalPowerProposition}, \ref{CaProposition}, \ref{CaastProposition},  and \ref{CamixProposition}. Moreover, $T_{B_E}$ is strongly continuous since one has that $\mathcal{D}(J_E)$ is dense in $E$ if and only if $E$ has order continuous norm (see e.g. \cite[Remark 4.2]{arendt2002spectrum}). Then, the assertions follow immediately by the relation between the well-posedness of a Cauchy problem, and the fact that $B_E$ generates a strongly continuous semigroup (see for example \cite[Proposition 3.1.2 and Theorem 3.1.12]{arendt2011vector}).
		
		Regarding the regularity result, one has that $u(t)$ is $E$-holomorphic in $t$ in $(0,\infty)$ since $T_{B_E}$ is a holomorphic semigroup. Even more, it satisfies  $u(t) = T_E (t) f$,  $u^{(k)} (t) = (B_E)^k u(t)$, and that $u^{(k)} (t) \in \mathcal{D}((B_E)^n)$ for all $k,n \in \NN$ and $t>0$ (see \cite[Chapter 3]{arendt2011vector}).  Now, reasoning as in the proof of Proposition \ref{domainProp} with any of the operators $B_E$ yields that $\mathcal{D}(B_E) \subset \mathcal D((J_E)^\varepsilon)$ for sufficiently small $\varepsilon >0$. In addition, since  $\mathcal{D}(J_E) \subset AC_{\rm loc}(0,\infty)$, we have that $\mathcal{D}((J_E)^{j+1}) \subset C^j(0,\infty)$. As $u^{(k)} (t) \in \mathcal{D}((B_E)^n) \subset \mathcal{D}((J_E)^{n\varepsilon})$ for all $k,n \in \NN$, one obtains that $u^{(k)}(t) \in C^\infty(0,\infty)$ for all $k\in \NN$ and $t>0$. The proof is finished.
	\end{proof}
	
	\begin{remark}
	As stated in the above proof, $T_{B_E}$ is strongly continuous at $0$ if and only if $E$ has order continuous norm. Hence,  Theorem \ref{ACP0Theorem} does not hold for a general $(L^1-L^\infty)$-interpolation space. To address all interpolation spaces, we follow the ideas given in \cite{arendt2002spectrum} and consider the K\"othe dual $E^\star$ of $E$, given by
	\begin{align*}
		E^\star := \left\{g:(0,\infty)\to\CC\mbox{ measurable and } \; \int_0^\infty |f(x)g(x)|\;dx < \infty \quad \text{for all } f \in E \right\}.
	\end{align*}
	Every $g \in  E^\star$ defines a bounded (order continuous) linear functional $\varphi_g$ on $E$, given by 
	$$\langle f, \varphi_g\rangle_{E,E^\star}:  = \int_0^\infty f(x)g(x)\;dx\;\mbox{ for all }\; f \in E.$$
	In this way we can identify $E^\star$ with a subspace of the dual $E'$ (and under the present assumptions on $E$, this subspace is norming for $E$).  It is known that, when equipped with the norm $\|g\|_{E^\star} = \|\varphi_g\|_{E'}$, then $E^\star$ is a $(L^1,L^\infty)$-interpolation space on $(0,\infty)$.
	\end{remark}
	
	
	Next, we consider the following abstract Cauchy problem:
	\begin{align}\label{CauchyP1}
		\tag{$ACP_1$}
		\begin{cases*}
			\displaystyle u \in C^1((0,\infty);E), \quad  u(t) \in \mathcal{D}(B_E), \quad t>0,
			\\ \displaystyle u' (t) =  B_E u (t), \qquad\,\,\, t>0,
			\\  \displaystyle \lim_{t\downarrow 0}\langle u(t), \varphi\rangle_{E,E^\star} = \langle f, \varphi \rangle_{E,E^\star}, \quad f \in E \text{ and for all } \varphi \in E^\star.
		\end{cases*}
	\end{align}
	Again, we  say that ($ACP_1$) is well-posed if, for any $f \in E$, there exists a unique $u$ which is a solution of ($ACP_1$).
	
	We have the following result.
	
	\begin{theorem}\label{ACP1Theorem}
		Let $E$ be a $(L^1-L^\infty)$-interpolation space,  $n \in \NN$,  and $\alpha >0$. Then,  the following assertions hold.
		\begin{enumerate}[(a)]
			\item  If $\alpha \in \left(n-\frac{1}{2}, n+\frac{1}{2}\right)$, then ($ACP_1$) is well-posed with $B_E = (-1)^{n+1} (J_E)^{2\alpha}$.
			\item If  $\overline{\eta}_E<1$ and if $\alpha \in \left(n-\frac{1}{2}, n+\frac{1}{2}\right)$, then ($ACP_1$) is well-posed with $B_E = (-1)^{n+1} (I-\mathcal{\riemannDerivative}_E^\alpha)^2$.
			\item If $\underline{\eta}_E>0$ and if $\alpha \in \left(n-\frac{1}{2}, n+\frac{1}{2}\right)$, then ($ACP_1$) is well-posed with $B_E = (-1)^{n+1}(\mathcal{W}_E^\alpha)^2$.
			\item If $\overline{\eta}_E<1$ and $\underline{\eta}_E > 0$, then ($ACP_1$) is well-posed with $B_E = \mathcal{W}_E^\alpha (I-\mathcal{\riemannDerivative}_E^\alpha)$.
		\end{enumerate}
		In any case, the solution $u$ of ($ACP_1$) is given by $u(t) = T_{B_E}(t) f$ for $t > 0$. In addition, identifying $u(t,x) = u(t)(x)$, we obtain that $u \in C^\infty((0,\infty) \times (0,\infty))$.
	\end{theorem}
	
	To prove the theorem, we need the following lemma.
	
\begin{lemma}\label{dualContinuity}
		Let $a\geq 0$ and let $A \in \BSect(\pi/2,a)$ on $E$ be such that $A$ generates an exponentially bounded  group $(\genericGroup(t))_{t\in \RR}$ for which $\|\genericGroup(t)\| \lesssim e^{a|t|}$ for $t \in \RR$. Let $\finalfunct \in \mathcal{M}_A$ satisfy all the hypothesis in Corollary \ref{holomorphicCor}. Assume furthermore that the following hold:
		\begin{enumerate}[(a)]
			\item $\finalfunct$ is quasi-regular in $\{-a,a,\infty\}$ with $\finalfunct(a), \finalfunct(-a) \neq \infty$.
			\item The group $(\genericGroup(t))_{t\in \RR}$ is $\sigma(E,E^\star)$-continuous, that is,  $\lim_{t\to0}\langle \genericGroup(t) f, \varphi\rangle_{E,E^\star} = \langle f, \varphi \rangle_{E,E^\star}$ for all $f \in E$ and $\varphi \in E^\star$.
		\end{enumerate}  
		Then, the semigroup $(T_\finalfunct(t))_{t \geq 0}$ generated by the operator $-\finalfunct(A)$ is also $\sigma(E,E^\star)$-continuous, i.e.,
		\begin{equation*}
			\lim_{t\downarrow 0}\langle T_\finalfunct(t) f, \varphi\rangle_{E,E^\star} = \langle f, \varphi \rangle_{E,E^\star} \quad \text{for all } f \in E \text{ and  } \varphi \in E^\star. 
		\end{equation*}
	\end{lemma}
	
	\begin{proof}
		We ask for the regularity conditions at $\{-a,a,\infty\}$ instead of just $M_A$ in order to apply the results given in the appendix of this paper. Now, Proposition \ref{HilleStrip} yields that $T_{\finalfunct}(t) = \int_{-\infty}^\infty \genericGroup(s)\, \mu_{h_t} (ds)$, where $h_t(z) := \exp(-t\finalfunct(z))$,  and $\mu_{h_t} \in M_a(\RR)$ is given in Lemma \ref{FourierofF}. By Lemma \ref{FourierofF} again, one obtains that 
		$$\int_{-\infty}^\infty \,\mu_t (ds) = \int_{-\infty}^\infty e^{i0s}\,\mu_t (ds) = h_t(0) = \exp(-t\finalfunct(0)).$$
		Then, for any $f \in E$ and $\varphi \in E^\star$, we have that
		\begin{align}\label{integra}
			\langle T_\finalfunct(t)f, \varphi \rangle_{E,E^\star} - \langle f, \varphi \rangle_{E,E^\star} = \int_{-\infty}^\infty \langle \genericGroup(s)f - f, \varphi\rangle_{E,E^\star}  \mu_{h_t}(s) \, ds
			+ (e^{-t\finalfunct(0)}-1)\langle f, \varphi \rangle_{E,E^\star}.
		\end{align}
		We have to prove that the integral term in \eqref{integra} tends to $0$ as $t \downarrow 0$. Since by assumption $(\genericGroup(t))_{t\in \RR}$ is $\sigma(E,E^\star)$-continuous, we have that $\lim_{t\downarrow 0} \langle \genericGroup (t) f, \varphi \rangle_{E,E^\star} = \langle f, \varphi \rangle_{E,E^\star}$. Thus, for any $\varepsilon > 0$, there exists $\delta > 0$ such that $|\langle \genericGroup(s)f - e^{irs}f, \varphi \rangle_{E,E^\star} | < \varepsilon$ for all $|s| < \delta$. Hence, for some $C>0$ independent of $t$ and $\varepsilon$, we have that 
		\begin{align*}
			&\limsup_{t \downarrow 0}  \left|\int_{-\infty}^\infty \langle \genericGroup(s)f - f, \varphi\rangle_{E,E^\star}  \mu_{h_t}(s) \, ds\right| 
			\leq  C\varepsilon  + \limsup_{t \downarrow 0}  \left|\int_{|s|>\delta} \langle \genericGroup(s)f - f, \varphi\rangle_{E,E^\star}  \mu_{h_t}(s) \, ds\right|.
		\end{align*}
		Let us work with the above integral when $s>\delta$, leaving the case $s<-\delta$, which is completely analogous.
		By Lemma \ref{FourierofF}, one gets that
		\begin{align*}
			&\int_{s>\delta} \langle \genericGroup(s)f - f, \varphi\rangle_{E,E^\star}\,  \mu_{h_t}(s) \, ds
			=  \int_{s>\delta} \langle \genericGroup(s)f - f, \varphi\rangle_{E,E^\star}\,  \frac{1}{2\pi i} \int_{\Gamma_+}  e^{-zs} e^{- t\finalfunct(z)} \, dzds
			\\ =& \int_{s>\delta} \langle \genericGroup(s)f - f, \varphi\rangle_{E,E^\star}  \frac{1}{2\pi i} \int_{\Gamma_+}  e^{-zs} \left(e^{- t\finalfunct(z)}- e^{-t\finalfunct(a)}\frac{b+a}{b+z}\right) \, dzds,
		\end{align*}
		where we have used Cauchy's theorem in the last equality and the fact that for $b>a$,  we have that
		$$e^{-t\finalfunct(a)} \int_{\Gamma_+}  e^{-zs}\frac{b+a}{b+z} \, dz = 0, \qquad \text{for all } s,t> 0.$$
		Now, applying the Lebesgue Dominated Convergence Theorem we obtain that 
		\begin{align*}
			&\limsup_{t \downarrow 0} \left|\int_{s>\delta} \langle \genericGroup(s)f - f, \varphi\rangle_{E,E^\star}\,  \mu_{h_t}(s) \, ds\right|
			=  \left| \int_{s>\delta} \langle \genericGroup(s)f - f, \varphi\rangle_{E,E^\star}  \frac{1}{2\pi i} \int_{\Gamma_+}  e^{-zs} \frac{z-a}{b+z} \, dzds \right| = 0,
		\end{align*}
		where we have used again Cauchy's theorem in the last equality. To check the hypothesis of the Dominated Convergence Theorem, one has to bound the following expression: 
		\begin{align*}
			F_t(s,z):= e^{-s(\Re z-a)}  \left|e^{- t\finalfunct(z)}- e^{-t\finalfunct(a)}\frac{b+a}{b+z}\right|, \quad s>\delta, \, z \in \Gamma_+,
		\end{align*}
		by an integrable function for all $t \in (0,\varepsilon')$, where $\varepsilon'$ is any number $\varepsilon' > 0$.	Since $\Re \finalfunct(z) \geq 0$ implies that $\sup_{t>0, z \in \Gamma_+}|e^{-t\finalfunct(z)}| < \infty$, an easy bound of the term between $|(\cdot)|$ leads to
		\begin{align*}
			F_t(s,z) \lesssim e^{-s(\Re z -a)} \min \left\{1,\frac{|z-a|+ |\finalfunct(z)-\finalfunct(a)|}{|b+z|}\right\},
		\end{align*}
		which is easily seen to be integrable by integrating first on $s$ and then in $z$ (recall that the function $\finalfunct$ is regular at $a$).
		The proof is finished.
	\end{proof}
	
	\begin{proof}[\bf Proof of Theorem \ref{ACP1Theorem}]
		Once we have proven that $T_{B_E}(t) f$ is $\sigma(E, E^\star)$-continuous on $t$ as $t \downarrow 0$ for all $f \in E$, the assertions follow by a similar reasoning as in the proofs of Theorem \ref{ACP0Theorem} and \cite[Theorem 5.8]{arendt2002spectrum}. Then, for the operators we are considering, we only have to check the exponentially bound condition of Lemma \ref{dualContinuity}. But, except for the case $\underline{\eta}_E = 0$ and $B_E = (J_E)^{2\alpha}$, we can always assume that they are satisfied since the functions $\finalfunct_{B_E}$,  for which $B_E = \finalfunct_{B_E}(J_E)$,  are holomorphic in strictly wider bisectors than the ones with singular points in $\underline{\eta}_E, \overline{\eta}_E$. If this is the case, given any $\varepsilon >0$, we have that $\|G_E(t)\|_{\mathcal L(E)} \lesssim \max\{e^{(\underline{\eta}_E - \varepsilon) t}, e^{(\overline{\eta}_E+\varepsilon) t}\}$ for all $t \in \RR$. And regarding the case $\underline{\eta}_E = 0$, one still has that $\|G_E(t)\|_{\mathcal L(E)}\lesssim 1$ for $t\leq 0$ (see e.g.  \cite{arendt2002spectrum}). Then, we can apply Lemma \ref{dualContinuity} to obtain that $T_{B_E}(t)f$ is $\sigma(E,E^\star)$-continuous. The proof is finished.
	\end{proof}
	
	It is easy to check that when $\alpha = 1$, all the different generalized Black--Scholes equations presented above yield the classical  Black--Scholes equation given by \eqref{BSEquationIntro}. In this case,  the above results  retrieve the ones obtained in \cite[Section 5]{arendt2002spectrum}. In particular, one gets the formula for the semigroup $T_{B_E}$, given by 
	\begin{align*}
		\left(T_{B_E}(w)f\right)(x) &= \frac{1}{2\pi} \int_0^\infty \frac{f(s)}{s} 
		\int_{-\infty}^\infty \left(\frac{s}{x}\right)^{iu} \exp\left(-wu^2\right) \, du ds 
		\\ &= \frac{1}{\sqrt{4\pi w}} \int_0^\infty \exp \left(- \frac{(\log x - \log s)^2}{4w}\right) \frac{f(s)}{s}\,ds, \qquad x >0, \, \Re w > 0,
	\end{align*}
	where in the last equality we have made use of the integral identity \cite[Formula 3.233(2)]{gradshteyn2014table}.

	\begin{remark}
		The above results do not cover (in general) the case $\alpha = 1/2,3/2,5/2,...$. This is closely related to the odd powers of a generator of a group (see Corollary \ref{StripCor} and \cite[Theorem 4.6]{baeumer2009unbounded}).  Indeed, one can prove that when $\alpha = 1/2,3/2,5/2,...$, the considered operators for $B_E$, except the last one $\mathcal{W}_E^\alpha (I - \mathcal{\riemannDerivative}_E^\alpha)$,  are bisectorial operators of angle $\frac{\pi}{2}$. Unfortunately, this is a necessary but not sufficient condition to determine that they generate semigroups.
	\end{remark}

	
{\appendix\section{Functional calculus of generators of exponentially bounded groups}\label{Appendix}
	In this appendix, we give some auxiliary results in the case where $A$ is the generator of an exponentially bounded group $(\genericGroup(t))_{t\in \RR}$ on a Banach space $X$ satisfying $\|\genericGroup(t)\|_{\mathcal L(X)} \lesssim \exp(a|t|)$ for all $t \in \RR$ and some $a\geq 0$. It is well known that in this case, $A \in \BSect(\pi/2, a)$, see for example \cite[Section 2.1.1]{haase2006functional}.  The following results are completely analogous to the ones given in \cite[Theorem 5.2]{bade1953operational} for the primary functional calculus of strip operators; or in \cite[Section 3.3]{haase2006functional} for the NFC of sectorial operators.
	
	It should be mentioned that through this appendix, we will only work with the NFC for bisectorial rate operators, not including the different NFCs presented in Section \ref{extensionCalculusSubsection}. The reason for this is that, in order to successfully apply some identities, we will need that the integration paths of the NFC leave the spectrum of $A$ completely on one side. This is enough to cover all the results that use the appendix.

	First, recall the integral representation for the resolvent of the generator $A$ of an exponentially bounded group $(\genericGroup(t))_{t\in \RR}$ given by
	\begin{equation}
		\begin{aligned}\label{integralRepresentationResolvent}
			R(z,A) = \int_0^\infty e^{-zt}\genericGroup(t) \, dt, \text{ if }\Re z >a, \quad R(z,A) &= -\int_{-\infty}^0 e^{-zt}\genericGroup(t) \, dt, \text{ if }\Re z <-a.
		\end{aligned}
	\end{equation}
	
	Next, for any $a \geq 0$, let $M_a (\RR)$ be the set of Borel measures $\mu$ on $\RR$ for which $e^{a|t|}$ is $\mu$-integrable. It is readily seen that $M_a (\RR)$ is closed under translation and convolution. Moreover, for any $\mu \in M_a (\RR)$, one can define its Fourier transform $\mathcal{F}$ given by
	\begin{align*}
		(\mathcal{F} \mu) (z) = \int_{-\infty}^\infty e^{-zt} \,\mu(dt), \qquad \text{for all } z \in BS_{\pi/2,a}.
	\end{align*}
	

	\begin{lemma}\label{FourierofF}
		Let $a\geq0$ and $f \in \mathcal{E}[BS_{\pi/2,a}]\oplus \CC \mathbf{1}$. Then,  there exists a (unique) measure $\mu_f \in M_a(\RR)$ such that $f(z) = \mathcal{F}\mu_f(-z)$ for all $z \in BS_{\pi/2,a}$, which is given by $\mu_f(dt) = \psi_f (t) dt + c\delta_0(dt)$, where $c = f (\infty)$ and
		\begin{align}\label{psifFormula}
			\psi_f (t) := \begin{cases*}
				\displaystyle	\frac{-1}{2\pi i} \int_{\Gamma_-} e^{-zt}f(z) \, dz, \quad t < 0,\\
				\displaystyle \frac{1}{2\pi i} \int_{\Gamma_+} e^{-zt} f(z) \, dz, \quad t > 0,
			\end{cases*}
		\end{align}
		and where $\Gamma$ is any path of integration for the NFC of bisectorial operators, $\Gamma_- := \Gamma \cap \Re z < -a$ and $\Gamma_+ := \Gamma \cap \Re z > a$.
	\end{lemma}
	
	\begin{proof}
		The proof is the same as in the case of sectorial operators (see \cite[Lemma 3.3.1]{haase2006functional}).  We omit the details for the sake of brevity.
	\end{proof}
	
	\begin{remark}\label{psiFdecayingRemark}
		Let $f$ be as above, and assume furthermore that $|f(z)| \lesssim |z|^{-(1+\varepsilon)}$ as $z \to \infty$ for some $\varepsilon > 0$. An easy application of Cauchy's theorem to \eqref{psifFormula} yields that
		\begin{align*}
			\psi_f (t) =  \frac{1}{2\pi} \int_{-\infty}^\infty e^{-itu}f(iu) \, du, \quad t \in \RR.
		\end{align*}
	\end{remark}


	\begin{proposition}\label{HilleStrip}
		Let $A$ be the generator of an  exponentially bounded group $(\genericGroup(t))_{t\in \RR}$ on $X$ satisfying $\|\genericGroup(t)\|_{\mathcal L(X)} \lesssim  e^{a |t|}$ for some $a\geq 0$, so that $A \in \BSect(\pi/2,a)$. Let $\mu \in M_a(\RR)$ be such that $f(z):=\mathcal{F}\mu(-z) \in \mathcal{M}[BS_{\pi/2,a}]_A$. Then, 
		$$f(A) = \int_{-\infty}^\infty \genericGroup(t) \, \mu(dt).
		$$
	\end{proposition}
	
	\begin{proof}
		The proof follows as in the case of sectorial operators (see \cite[Proposition 3.3.2]{haase2006functional}). We omit the details for the sake of brevity.
	\end{proof}
	
}	


\bibliographystyle{plain}
\bibliography{biblio}
\end{document}